\documentclass[bj,preprint]{imsart}

\usepackage{xcolor}
\usepackage{amsmath}
\usepackage{amsfonts}
\usepackage{amsthm}
\usepackage{amssymb}
\usepackage{graphicx}
\usepackage[colorlinks=true,linkcolor=blue]{hyperref}
\usepackage{blkarray}
\usepackage{mathrsfs}
\usepackage[breakable, theorems, skins]{tcolorbox}
\usepackage{xargs}
\usepackage{enumerate}

\usepackage{ifthen}
\newcommand{\isdraft}{\boolean{true}} 
\renewcommand{\isdraft}{\boolean{false}} 
\ifthenelse{\isdraft}{
    \usepackage[color]{showkeys}
    \usepackage{xcolor}
}{}
\newcommand{\markupdraft}[2]{
    \ifthenelse{\equal{#1}{display}}{#2}{}
    \ifthenelse{\equal{#1}{color}}{\color{#2}}{}
}
\newcommand{\nnotecolored}[3][]{\markupdraft{display}{{\color{#2}\noindent[{\bf Note #1}: #3]}}}
\newcommand{\mathnote}[1]{\markupdraft{display}{{\color{brown}\noindent[{\bf Math note}: #1]}}}
\newcommand{\newcolored}[3][]{{\markupdraft{color}{#2}#3}
    \ifthenelse{\equal{#1}{}}{}{\markupdraft{display}{{\color{yellow!70!black}[#1]}}}} 

\providecommand{\del}[2][]{{\markupdraft{display}{{\color{red!20!yellow}[rmed: "#2"[#1]]}}}} 
\providecommand{\new}[2][]{\newcolored[#1]{blue}{#2}}
\providecommand{\rem}[2][]{\nnotecolored[#1]{green}{#2}} 
   
\providecommand{\TODO}[2][]{\markupdraft{display}{{\color{red}~\noindent== TODO: #2 {\color{yellow}(#1)} ==}}}

\providecommand{\todo}[2][]{\markupdraft{display}{{\color{red}\noindent++TODO: #2 {\color{yellow}(#1)}++}}}

\ifthenelse{\isdraft}{}{\renewcommand{\markupdraft}[2]{}}

\startlocaldefs

\newcommand{\anne}[1]{\rem[Anne]{\color{orange}#1}}
\newcommand{\alex}[1]{\rem[Alex]{\color{green}#1}}

\newcommand{\bs}[1]{\ensuremath{\mathbf{#1}}}
\newcommand{\markov}{\ensuremath{\Phi}}
\newcommand{\pt}[1][t]{\ensuremath{\Phi_{#1}}}

\newcommand{\pk}[1][k]{\Phi_{#1}}
\newcommand{\pkk}{\pk[k+1]}
\newcommand{\W}{{W}}
\newcommand{\w}{\bs{w}}
\newcommand{\uu}{\bs{u}}
\renewcommand{\u}{\bs{u}}
\renewcommand{\v}{\bs{v}}

\newcommand{\U}{U}
\newcommand{\Ut}[1][k]{\ensuremath{{\U}_{#1}}}
\newcommand{\Utt}{\Ut[k+1]}
\newcommand{\Uk}[1][k]{\U_{#1}}
\newcommand{\Ukk}{\Uk[k+1]}

\newcommand{\R}{\ensuremath{\mathbb{R}}}
\newcommand{\N}{\ensuremath{\Zplus}}
\newcommand{\borel}{\mathscr{B}}
\newcommand{\mleb}{\mu^{\textrm{Leb}}}
\newcommand{\1}{\mathbf{1}}

\newcommand{\weight}{\beta}
\newcommand{\ZZ}{{\bs {\cal Z}}}

\newcommand{\x}{x}
\newcommand{\X}{{X}}
\newcommand{\Xt}[1][k]{\X_{#1}}
\newcommand{\Xtt}{\Xt[k+1]}

\newcommand{\perm}{\mathcal{S}}

\newcommand{\y}{{y}}

\newcommand{\st}[1][k]{\sigma_{#1}}
\newcommand{\Idn}{\textrm{I}_n}

\newcommand{\stt}{\st[k+1]}
\newcommand{\Z}{{Z}}
\newcommand{\Zt}[1][k]{\Z_{#1}}
\newcommand{\Ztt}{\Zt[k+1]}

\newcommand{\z}{\bs{z}}

\newcommand{\0}{\bs{0}}

\newcommand{\xdom}{\mathcal{X}}
\newcommandx{\Oxk}[2][1=k,2=x]{{\mathscr O}_{#2}^{#1}}
\newcommandx{\pxk}[2][1=k,2=x]{p_{#2}^{#1}}

\newcommandx{\Sxk}[2][1=k,2=x]{S_{#2}^{#1}}

\newcommand{\xstar}{x^*}
\newcommand{\wstar}{\w^*}
\newcommand{\ystar}{y^*}

\newcommand{\Uopen}{\mathscr{U}}
\newcommand{\Vopen}{\mathscr{V}}

\newcommand{\Wtt}{\W_{k+1}}

\newcommand{\OOpen}{\mathscr{O}}

\newcommand{\Oo}{\mathscr{O}}
\newcommand{\Ro}{\mathscr{R}}
\newcommand{\Uo}{\mathscr{U}}
\newcommand{\Vo}{\mathscr{V}}
\newcommand{\Xo}{\mathscr{X}}
\newcommand{\Wo}{\mathscr{W}}

\newcommand{\Zplus}{\mathbb{Z}_{\geq 0}}
\newcommand{\Zplusstrict}{\mathbb{Z}_{> 0}}

\newcommandx{\Cxk}[2][1=k,2=x]{C_{#2}^{#1}}

\newtheorem{proposition}{Proposition}[section]
\newtheorem{lemma}{Lemma}[section]
\newtheorem{theorem}{Theorem}[section]

\newtheorem{corollary}{Corollary}[section]

\newtheorem{definition}{Definition}[section]
\newtheorem{example}{Example}
\newtheorem{remark}{Remark}

\DeclareMathOperator{\supp}{supp}
\DeclareMathOperator{\CM}{CM}
\DeclareMathOperator{\rank}{rank}

\endlocaldefs

\begin{document}

\begin{frontmatter}

\title{
Verifiable Conditions for the Irreducibility and Aperiodicity of Markov Chains by Analyzing Underlying Deterministic  Models
}

\runtitle{Conditions for Irreducibility and Aperiodicity via Underlying Deterministic  Models}

\begin{aug}
\author{\fnms{Alexandre} \snm{Chotard}\thanksref{a}\corref{bla}\ead[label=e1]{alexandre.chotard@gmail.com}}

\author{\fnms{Anne} \snm{Auger}\thanksref{b}\ead[label=e2]{anne.auger@inria.fr}}

\address[a]{Departement of Mathematics; 
KTH Royal Institute of Technology 
SE-10044 Stockholm - Sweden \printead{e1}}
\address[b]{RandOpt Team - Inria Saclay - \^{I}le-de-France, 
Ecole Polytechnique, Palaiseau - France 
\printead{e2}}

\runauthor{A. Chotard, A. Auger}


\end{aug}

\todo{\bf The new finish line!\\
\todo{link with 7.2.2 meyn tweedie}

}

\begin{abstract}
We consider Markov chains that obey the following general non-linear state space model: 
$\pt[k+1] = F(\pt[k], \alpha(\pt[k], U_{k+1}))$ where the function $F$ is $C^1$ while $\alpha$ is typically discontinuous and $\{U_k: k \in \Zplusstrict \}$ is an independent and identically distributed process. We assume that for all $x$, the random variable $\alpha(x, U_1)$ admits a density $p_x$ such that $(x, w) \mapsto p_x(w)$ is lower semi-continuous.

We generalize and extend previous results that connect properties of the underlying deterministic control model to provide conditions for the chain to be $\varphi$-irreducible and aperiodic. By building on those results, we show that if a rank condition on the controllability matrix is satisfied for all $x$,  there is equivalence between the existence of a globally attracting state for the control model and $\varphi$-irreducibility of the Markov chain. Additionally, under the same rank condition on the controllability matrix, we prove that there is equivalence between the existence of a steadily attracting state and the $\varphi$-irreducibility and aperiodicity of the chain. The notion of steadily attracting state is new. We additionally derive practical conditions by showing that the rank condition on the controllability matrix needs to be verified only at a globally attracting state (resp.\ steadily attracting state) for the chain to be a $\varphi$-irreducible T-chain (resp.\ $\varphi$-irreducible aperiodic T-chain).

Those results hold under considerably weaker assumptions on the model than previous ones that would require  $(x,u) \mapsto F(x,\alpha(x,u))$ to be $C^\infty$ (while it can be discontinuous here). Additionally the establishment of a \emph{necessary and sufficient} condition for the $\varphi$-irreducibility and aperiodicity without a structural assumption on the control set is novel---even for Markov chains where $(x,u) \mapsto F(x,\alpha(x,u))$ is $C^\infty$. \anne{I add here the structural assumption - because of the connected control set setting}

We illustrate that the conditions are easy to verify on a non-trivial and non-artificial example of Markov chain arising 
in the context of adaptive stochastic search algorithms to optimize continuous functions in a black-box scenario.

\end{abstract}

\end{frontmatter}



\section{Introduction}\label{sec:intro}

Markov chain theory is widely applied for analyzing methods arising in different domains like machine learning, time series analysis, statistics or optimization.
Prior to establishing stability properties like geometric ergodicity or using sample path theorems, one often needs to prove basic standard properties such as $\varphi$-irreducibility and aperiodicity (see for instance the Law of Large Numbers in \cite[Theorem 17.0.1]{meyntweedie} or the aperiodic and geometric ergodic theorems \cite[Theorem 13.0.1 and 15.0.1]{meyntweedie}). \alex{$\varphi$-irreducibility is critical to the theory and is needed in many results, for example in Theorem 8.0.1 which guarantees that a Markov chain is either recurrent or transient.}
In addition, to prove the existence of an invariant probability  distribution or geometric ergodicity, it is often practical to use \emph{drift} conditions that roughly speaking state that outside a specific set, the conditional expected progress measured in terms of an appropriate non-negative potential function should be negative. The specific sets are typically so-called \emph{small sets} and they need to be identified in order to prove a drift condition.


Establishing $\varphi$-irreducibility, aperiodicity and identifying small sets can turn out to be very challenging. In the domain of time series analysis, this observation was already done and several works developed tools to facilitate this task. This includes\alex{specific but verifiable} conditions on the model parameters of specific time series (e.g.\ bilinear models~\cite{pham1986mixing} or non linear autoregressive time series~\cite{cline1998irreducibility}), general conditions on the underlying deterministic control model~\cite{meyncaines1989}, \cite[Chapter~7]{meyntweedie} or on small or petite sets~\cite{cline2011irreducibility,chan1990tong}. One notable work in the latter direction is presented in \cite{cline2011irreducibility} where the equivalence between $\varphi$-irreducibility, aperiodicity and the $T$-chain property with conditions on reachable petite or small sets has been shown. Remarkably, these results hold under weak conditions (at most weak-Feller) for bounded positive kernels. However, in practice it can be difficult to show that a set is small or petite without having first shown that the Markov chain is a $T$-chain. Hence these conditions, albeit weak, can turn out to be difficult to verify. Small sets can be shown to exist using a rank condition on the controllability matrix as shown in \cite[Proposition~2.1]{meyn-caines1991} or Proposition~\ref{pr:2.1} in this paper, and then the results obtained in \cite{cline2011irreducibility} are very similar to the ones presented here.

For Markov chains following a non-linear state space model of the form
\begin{equation} \label{eq:gmodel}
\pkk = G(\pk, \Ukk ),  k \in \N,
\end{equation}
where $G: \xdom \times \R^p \to \xdom$ with $\xdom \subset \R^n$ open is $C^\infty$ and $\{ U_k : k \in \Zplusstrict \}$ is an independent identically distributed (i.i.d.) process, independent of $\Phi_0$, some practical tools for proving $\varphi$-irreducibility, aperiodicity and identify that compact are small sets rely on investigating the underlying deterministic \emph{control model}~\cite[Chapter~7]{meyntweedie}, \cite{meyncaines1989,meyn-caines1991}. They connect structural and stability properties of the deterministic control model to structural and stability aspects of the associated Markov chains.\alex{Stability is probably misguiding here (I would be thinking of invariant measure/ergodicity). }\anne{I took this phrasing in the Meyn - Tweedie ... If we put only structural, it looks wrong. Typically existence of a globally attracting state is really a notion of stability. Think about dynamic systems for instance ...} In contrast to the works mentioned above \cite{cline2011irreducibility,chan1990tong}, the particularly attracting feature is that the tools entail manipulating deterministic sequences of possible paths followed by the underlying deterministic algorithm and are thus relatively straightforward to verify, at the cost of being constrained to the model \eqref{eq:gmodel} where in particular $G$ is $C^\infty$. 

The assumption that $G$ is $C^\infty$ is quite restrictive for some non-linear state space models.
Particularly, many Markov chains arising in the context of adaptive randomized optimization algorithms are associated to functions $G$ that are discontinuous. 
Yet, we show in this paper that most of the results presented in \cite[Chapter~7]{meyntweedie} holding for chains following \eqref{eq:gmodel} generalize to a broader model presented below---that naturally arises in the context of adaptive comparison-based algorithms.
More precisely, we consider Markov chains following the model
\begin{equation} \label{eq:fmodel}
\pkk = F(\pk , \alpha(\pk , \Ukk)),  k \in \N ,
\end{equation}
where for all $k$, $\pk \in \xdom$ with $\xdom$ an open subset of $\R^n$, $U_k \in \R^m$ with $\{ \Uk: k \in \Zplusstrict\} $ an i.i.d. process,  $F: \xdom \times \R^p \to \xdom$ is a $C^1$ function, $\alpha : \xdom \times \R^m \to \R^p$ is a measurable function---that can typically be discontinuous---and for all $x \in \xdom$, the distribution $\mu_{x}$ of the random variable $\alpha(x, \U_1)$ admits a density $p_x()$ such that the function $(x, w) \mapsto p_x(w)$ is lower semi-continuous with respect to both variables. In the models \eqref{eq:gmodel} and \eqref{eq:fmodel}, $n, m, p$ belong to $\Zplusstrict$. Note that \eqref{eq:gmodel} and \eqref{eq:fmodel} are linked via the relation $G(x,u)=F(x,\alpha(x,u))$, so discontinuities of $\alpha$ may render $G$ discontinuous, hence not satisfying the smooth condition of \cite{meyntweedie}, while $F$ itself can be $C^1$ or smooth. Remark that model \eqref{eq:fmodel} is actually a strict generalization of \eqref{eq:gmodel}: by setting $\alpha(x,u) = u$ in \eqref{eq:fmodel} we indeed recover \eqref{eq:gmodel} such that we will talk about a single model---model \eqref{eq:fmodel}---that can reduce to \eqref{eq:gmodel} if $\alpha(x,u)=u$.



\newcommand{\Normal}{\mathcal{N}(0,1)}
\paragraph{Toy Examples} To motivate the general model \eqref{eq:fmodel}, consider first an additive random walk on $\R$ defined by choosing for $\Phi_0$ an arbitrary distribution on $\R$  and for all $k \in \Zplus$ by
\begin{equation}\label{eq:RW}
\pkk = \pk + U_{k+1}
\end{equation}
where $\{ U_k : k \in \Zplusstrict \} $ is an independent and identically distributed Gaussian process with each $U_k$ distributed as a standard normal distribution, that is $U_k \sim \Normal$ for all $k$. This Markov chain follows the model \eqref{eq:gmodel} with $G(x,u)=x+u$ and hence model \eqref{eq:fmodel} with $F(x,u)=x+u$ and $\alpha(x,u)=u$. The random variable $\alpha(x,U_1)$ admits the density
\begin{equation}\label{eq:densityToy1}
p(w) = p_{\mathcal{N}}(w) : = \frac{1}{\sqrt{2 \pi}} \exp( - w^2/2)  .
\end{equation}
\anne{This chain is strong Feller.}

We describe now a variation of this Markov chain where the update of $\pk$ is also additive and more precisely $F(x,w)=x+w$. 
Consider indeed a simple (naive) iterative optimization algorithm on $\R$, aiming at \emph{minimizing} an objective function $f: \R \to \R$ without using any derivatives of $f$. The algorithm delivers at each iteration $k$, an estimate of the optimum of the function encoded within the random variable $\Phi_k$. The first estimate $\Phi_0$ is sampled from an arbitrary distribution on $\R$. To obtain the estimate $\Phi_{k+1}$ from $\Phi_k$, two candidate solutions are sampled around $\Phi_k$
$$
\widetilde{\Phi}^i_{k+1} = \Phi_k + U_{k+1}^i, \mbox{ for } i=1, 2
$$
where for all $k \geq 1$, $\{U_{k}^i: i = 1, 2 \}$ are independent following each a normal distribution $\Normal$, and $\{ U_k=(U_k^1,U_k^2) : k \geq 1\}$ is an independent and identically distributed process. Those candidate solutions are evaluated on the objective function $f$ and ranked according to their $f$ values. A permutation that contains the index of the ordered candidate solutions is extracted, that is, $\perm$ is a permutation of the set $\{1,2\}$ such that\footnote{The unicity of the permutation can be guaranteed by imposing that if $f(\pt + U_{t+1}^{\perm(1)}) = f(\pt + U_{t+1}^{\perm(2)})$ then $\perm(1) = 1$.}
\begin{equation}\label{eq:selection}
f \left( \Phi_{k} + U_{k+1}^{\perm(1)} \right)  \leq f\left( \Phi_k + U_{k+1}^{\perm(2)} \right) .
\end{equation}
The new estimate of the solution corresponds to the best of the two sampled candidate solutions, that is 
\begin{equation}\label{eq:ex01}
\Phi_{k+1} = \Phi_k + U_{k+1}^{\perm(1)} .
\end{equation}
This algorithm is a simplification of some randomized adaptive optimization algorithms and particularly of evolution strategies \cite{book-chapter-niko-dirk-anne}. The update \eqref{eq:ex01} loosely drives $\Phi_k$ towards better solutions.  This algorithm is not meant to be a good optimization algorithm but serves as illustration for our general model. We will present more reasonable optimization algorithms in Section~\ref{sec:appli}.

Remark  that 
$
U_{k+1}^{\perm(1)} = ( U_{k+1}^1 - U_{k+1}^2) 1_{\{f(\Phi_k + U_{k+1}^{1}) \leq f(\Phi_k + U_{k+1}^2) \}} + U_{k+1}^2 
$
and define the function $\alpha: \R \times \R^2 \to \R$ as
\begin{equation}\label{eq:ex02}
\alpha(x,(u^1,u^2))= (u^1 - u^2) 1_{\{ f(x+u^1)\leq f(x+u^2) \}} + u^2
\end{equation}
such that
$
U^{\perm(1)}_{k+1} = \alpha(\Phi_k,(U_{k+1}^1,U_{k+1}^2))
$.
The update of $\Phi_k$ then satisfies
\begin{equation}\label{eq:example}
\Phi_{k+1} = \Phi_k + \alpha(\Phi_k,(U_{k+1}^1,U_{k+1}^2))  = F(\Phi_k,\alpha(\Phi_k,(U_{k+1}^1,U_{k+1}^2)))
\end{equation}
where $F(x,w) = x+w$. Suppose that $f(x)= x^2$, then the function $\alpha$ is discontinuous. (Indeed, for $u^1$ and $u^2$ different such that $f(x+u^1) = f(x+u^2)$, a small (continuous) change in $u^2$ can lead to $\alpha(x,(u^1,u^2))$ jumping from $u^1$ to $u^2$.) Hence, using the modeling via \eqref{eq:gmodel}, the corresponding function $G(x,u) = x + \alpha(x,u)$ with $u=(u^1,u^2)$ is discontinuous and does not satisfy the basic assumptions of the model such that the results presented in \cite[Chapter~7]{meyntweedie} cannot be directly applied. 


\newcommand{\NN}{\mathcal{N}}
For all $x$, the random variable $\alpha(x,U_1)$ admits a density equal to
\begin{equation}\label{eq:densityToy2}
p_x(w) = 2 p_\NN(w) \int 1_{\{(x+w)^2 < (x+u)^2 \}} p_\NN(u) du
\end{equation}
where $p_\NN(u) = \frac{1}{\sqrt{2 \pi}} \exp(- u^2/2)$ is the density of a standard normal distribution. The function $(x,w) \mapsto p_x(w)$ is continuous as a consequence of the Lebesgue dominated convergence theorem and hence lower semi-continuous. 

Those examples serve as illustration to the model underlying the paper. Yet $\varphi$-irreducibility, aperiodicity and identification of small sets for those Markov chains can be easily proven directly considering the transition kernel
\begin{equation}\label{eq:trans-kernel}
P(x,A) = \int 1_A(F(x+w)) p_x(w) d w =  \int 1_A( y) p_x(x-y) dy .
\end{equation}
We will present in Section~\ref{sec:appli} a more complex example---where $F$ is more complex and the density is lower semi-continuous but not continuous---where the tools developed in the paper are needed to easily prove $\varphi$-irreducibility, aperiodicity and identify small sets.

As sketched on this toy example, the function $\alpha(.,.)$ for defining a Markov chain following \eqref{eq:fmodel} naturally arises---in the context of randomized optimization---from the ``selection" of the step to update the state of the algorithm: in this simple example, we select the best step via \eqref{eq:selection}. This selection gives a discontinuous $\alpha$ function which leads to a  discontinuous underlying $G$ function.

While the Markov chains in \eqref{eq:RW} and \eqref{eq:example} look sensibly different, \new{we will see that} \del{we will show that} both chains have the same underlying deterministic (control) model and \del{that} consequently, the $\varphi$-irreducibility and aperiodicity of one chain implies the $\varphi$-irreducibility and aperiodicity of the other one. 

\paragraph{Summary of main contributions}

Our first contribution is to 
generalize the definition of the deterministic control model associated to a Markov chain following \eqref{eq:gmodel} \cite{meyncaines1989,meyn-caines1991,meyntweedie} to a Markov chain following \eqref{eq:fmodel} by extending the definition of the control set to a set of open sets---indexed by the initial conditions \new{and time steps}---where the extended lower semi-continuous densities $p_x^k()$ are strictly positive. Our first main result concerns the $\varphi$-irreducibility: We prove that under a proper controllability condition---formulated as a condition on the rank of the controllability matrix that needs to be satisfied for every $x$---the existence of a globally attracting state for the underlying deterministic control model is \emph{equivalent} to the Markov chain being $\varphi$-irreducible. This result generalizes to our context Proposition~7.2.6 presented in~\cite{meyntweedie}. It heavily relies on \cite[Theorem 2.1 (iii)]{meyn-caines1991} which can easily be transposed to our setting. We then establish a similar result for the $\varphi$-irreducibility and aperiodicity of the chain. We introduce the notion of steadily attracting state (implying global attractivity) and we prove that under the same controllability condition, the existence of a steadily attracting state is \emph{equivalent} to the $\varphi$-irreducibility and aperiodicity of the chain. This improves previous results derived for model \eqref{eq:gmodel} where the control set is assumed to be connected. 

We additionally derive practical conditions by showing that the rank conditions on the controllability matrix needs to be satisfied at a globally attracting state (resp.\ steadily attracting state) \emph{only} to imply the $\varphi$-irreducibility and T-chain property (resp.\ aperiodicity, $\varphi$-irreducibility and  T-chain property) of the Markov chain under the existence of a globally attracting (resp.\ steadily attracting) state. For the aperiodicity, our assumption of existence of a steadily attracting state is significantly weaker than the one of assymptotic controllability of the deterministic model made in \cite[Proposition~3.2]{meyncaines1989}.
We illustrate how to use the practical conditions on a non-trivial example arising in stochastic optimization.\\

This paper is organized as follows.  In Section~\ref{sec:background}, we introduce the Markov chain background necessary for the paper. In Section~\ref{sec:cm} we define the deterministic control model associated to the Markov chain and the different notions related to it.
%
In Section~\ref{sec:main} we present our main results related to $\varphi$-irreducibility and aperiodicity. In Section~\ref{sec:appli} we apply the results to the toy examples presented in this introduction and to a stochastic optimization algorithm.
%

\paragraph{Notations} We denote $\mathbb{Z}$ the set of integers
, $\Zplus$ the set of non negative integers $\{ 0, 1, 2, \ldots \}$ and $\Zplusstrict$ the set of \new{positive} integers \del{excluding zero} $\{ 1, 2, \ldots, \}$. We denote $\R_\geq$ the set of non-negative real numbers and $\R_>$ of  positive real numbers.
For $n \in \Zplusstrict$, we denote $\R^n$ the $n$-dimensional set of real numbers and $\mleb$ the $n$-dimensional Lebesgue measure. The Borel sigma-field of a topological space $\xdom$ is denoted $\borel(\xdom)$. For $x \in \R^n$ and $\epsilon >0$, $B(x, \epsilon)$ denotes the open ball of center $x$ and radius $\epsilon$. For $A \in \borel(\xdom)$, we denote $A^c$ the complement of $A$ in $\xdom$. 
For $u$ a real vector of $\R^n$, $u^T$ denotes the transpose of $u$. For $f$ and $g$ real-valued functions, following the Bachmann-Landau notation, $g = o(f)$ denotes that $g$ is a little-o of $f$.

\section{Background on Markov Chains} \label{sec:background}

In this paper, we consider the state-space $\xdom$ being an open subset of $\R^n$ and equipped with its Borel sigma-field $\borel(\xdom)$. A \emph{kernel} $K$ is a function on $(\xdom,\borel(\xdom))$ such that $ K(.,A)$ is measurable for all $A$ and for each $x \in \xdom$, $K(x,.)$ is a signed measure. 

A kernel $K$ is \emph{substochastic} if it is nonnegative and satisfies $K(x,\xdom) \leq 1$ \new{for all $x \in \xdom$}. It is a \emph{transition kernel} if it satisfies $K(x,\xdom)=1$ \new{for all $x \in \xdom$}.
Given a Markov chain $\Phi=\{\pk : k \in \Zplus \}$, we denote $P^k$, ${k \in \Zplusstrict}$, its \emph{$k$-step transition kernel} defined by
\begin{equation}
P^k(x,A) = \Pr( \Phi_k \in A | \Phi_0 = x )~, ~~x \in \xdom, ~A \in \borel(\xdom) .
\end{equation}
We shall be concerned with the question of \emph{$\varphi$-irreducibility}, that is whether there exists a non-trivial measure $\varphi$ on $\borel(\xdom)$ such that for all $A \in \borel(\xdom)$ with $\varphi(A) > 0$
$$
\sum_{k \in \Zplusstrict} P^k(x, A) > 0~, \textrm{ for all } x \in \xdom . 
$$
If such a $\varphi$ exists, the chain is called $\varphi$-irreducible. A $\varphi$-irreducible Markov chain admits a \emph{maximal irreducibility measure}, $\psi$, which dominates any other irreducibility measure, meaning for $A \in \borel(\xdom)$, $\psi(A) = 0 $ implies $\varphi(A)=0$ for any irreducibility measure $\varphi$ (see \cite[Theorem~4.0.1]{meyntweedie} for more details).

We also need the notion of T-chain defined in the following way.
Let $b$ be a probability distribution on $\Zplus$, and let $K_b$ denote the probability transition kernel defined by
\begin{equation}
K_b (x, A) := \sum_{k \in \N} b(k) P^k(x, A) ~,  x \in \xdom,  A \in \borel(\xdom) .
\end{equation}

Let $T$ be a substochastic transition kernel which satisfies
$$
K_b (x, A) \geq T(x, A)~,  \textrm{ for all } x \in \xdom,  A \in \borel(\xdom),
$$
and such that $T(\cdot, A)$ is a lower semi-continuous function for all $A \in \borel(\xdom)$. Then $T$ is called a \emph{continuous component} of $K_b$.
If a Markov chain $\markov$ admits a probability distribution $b$ on $\N$ such that $K_b$ possesses a continuous component $T$ satisfying $T(x, \xdom) > 0$ for all $x \in \xdom$, then $\markov$ is called a \emph{$T$-chain}.

A set $C \in \borel(\xdom)$ is called \emph{petite} if there exists $b$ a probability distribution on $\Zplus$ and $\nu_b$ a non-trivial measure on $\borel(\xdom)$ such that
$$
K_b (x, A) \geq \nu_b (A) ~,~~\textrm{for all } x \in C, ~A \in \borel(\xdom) .
$$
The set $C$ is then called a $\nu_b$-petite set.

Similarly, a set $C \in \borel(\xdom)$ is called \emph{small} if there exists $k \in \Zplusstrict$ and $\nu_{k}$ a non-trivial measure on $\borel(\xdom)$ such that
\begin{equation}
P^{k}(x,A) \geq \nu_{k}(A) ~,~~\textrm{for all } x \in C, ~ A \in \borel(\xdom).
\end{equation}
The set $C$ is then called a $\nu_{k}$-small set. Note that a $\nu_k$-small set is $\nu_{\delta_k}$-petite, where $\delta_k$ is the Dirac measure at $k$.

Last, we will derive conditions for a Markov chain following \eqref{eq:fmodel} to be aperiodic. We therefore remind the definition of an aperiodic Markov chain.
Suppose that $\markov$ is a $\varphi$-irreducible Markov chain. For $d \in \Zplusstrict$, let $(D_i)_{i = 1, \ldots, d} \in \borel(\xdom)^d$ be a sequence of disjoint sets. We call $(D_i)_{i = 1, \ldots, d}$ a \emph{$d$-cycle} if

$(\mathrm{i})$ $P(x, D_{i+1}) = 1$ for all $x \in D_i$ and $i = 0, \ldots, d-1$ (mod $d$),

$(\mathrm{ii})$ $\varphi\left( \left( \bigcup_{i=1}^{d} D_i \right)^c \right) = 0$ for all $\varphi$-irreducibility measure of $\markov$.

If $\markov$ is $\varphi$-irreducible, there exists a $d$-cycle with $d \in \Zplusstrict$~\cite[Theorem~5.4.4]{meyntweedie}. The largest $d$ for which there exists a $d$-cycle is called the \emph{period} of $\markov$. If the period of $\markov$ is $1$, then $\markov$ is called \emph{aperiodic}.

For more details on Markov chains theory we refer to \cite{meyntweedie,nummelin1984}.

\section{Deterministic Control Model: Definitions and First Results} \label{sec:cm}

From now on, we consider a Markov chain defined via \eqref{eq:fmodel}. 
We pose the following basic assumptions on the initial condition $\pk[0]$ and the disturbance process $\bs{U} := \{ \Uk : k \in \Zplusstrict \}$.
\begin{itemize}
\item [A1.] $(\pk[0], \bs{U})$ are random variables on a probability space $(\Omega, \mathcal{F}, P_{\pk[0]})$; 
\item [A2.] $\pk[0]$ is independent of $\bs{U}$;
\item [A3.] $\bs{U}$ is an independent and identically distributed process.
\end{itemize} 
We additionally assume that 
\begin{itemize}
\item [A4.] For all $x \in \xdom$, the distribution $\mu_{x}$ of the random variable $\alpha(x, \U_1)$ admits a  density $p_x()$, such that the function $(x,w) \mapsto p_x(w)$ is lower semi-continuous with respect to both variables;
\item[A5.] The function $F: \xdom \times \R^p \to \xdom$ is $C^1$.
\end{itemize}
\mathnote{The Vitali-Carathéodory theorem implies that for any $f \in L^1(\R^p, \borel(\R^p), dx)$ there exists a lower semi-continuous function $g \in L^1$ arbitrarily close to $f$ in $L^1$ norm. Also, in \cite{intelligentComputing} it is stated that any $f$ with finite total variation, there exists a lower semi-continuous function $g$ equal almost everywhere to $f$ }
\mathnote{A5 may be replaced by $F$ differentiable almost everywhere + $C^1$ on an open set containing the points $\xstar$ and $\wstar$ from the theorems.}

\subsection{Deterministic Control Model (CM(F))}
The attractive feature of the results presented in \cite[Chapter~7]{meyntweedie} for Markov chains following model \eqref{eq:gmodel} is that they entail manipulating deterministic trajectories of an underlying deterministic control model. The Markov chains considered in the present paper strictly generalize the Markov chains following \eqref{eq:gmodel} by simply assuming that $\alpha(x,u)=u$. We here generalize the underlying deterministic control model introduced for Markov chains following \eqref{eq:gmodel} to Markov chains following \eqref{eq:fmodel}.


We consider first the \emph{extended transition map} function  \cite{meyn-caines1991}  $\Sxk : \R^{pk} \to \xdom$ defined inductively for $k \in \Zplus$, $x \in \xdom$ and $\w = (w_1, \ldots, w_k) \in \R^{pk}$ by
\begin{align} \label{eq:sxk}
&\Sxk(\w) := F(\Sxk[k-1](w_1, \ldots, w_{k-1}), w_k),  k \in \Zplusstrict, \\
&\Sxk[0] := x . \nonumber
\end{align}
If the function $F$ is $C^k$ then the function $(x, \w) \mapsto \Sxk(\w)$ is also $C^k$ with respect to both variables (\del{see}Lemma~\ref{lm:scp}\del{for details}).

The extended probability density is the function $p_x^{k}$ defined inductively for all $k \in \Zplusstrict$, $x \in \xdom$ and $\w = (w_1, \ldots, w_{k}) \in \R^{pk}$ by
\begin{align} \label{eq:pxk}
&\pxk[k] (\w) := \pxk[k-1](w_1, \ldots, w_{k-1}) p_{S_x^{k-1}(w_1, \ldots, w_{k-1})}(w_{k}) \\
&\pxk[1] (w_1) := p_x(w_1). \nonumber 
\end{align}
Let $W_1 := \alpha(x, U_1)$ and $W_{k} := \alpha(\Sxk[k-1](W_1, \ldots, W_{k-1}), U_{k}) $ for all $k \in \Zplusstrict$, the extended probability function $p_x^k$ is a probability density function of $(W_1, W_2, \ldots, W_k) $. In the case where $\alpha(x, u) = u$, denoting $p$ the lower semi-continuous density of $U_1$ the extended probability density reduces to $p^k(\w)= p(w_1) \ldots p(w_k) $ (that is, in the case of a Markov chain that also follows \eqref{eq:gmodel}).
 
The function $(x, \w) \mapsto p_x^k(\w)$ is lower semi-continuous as a consequence of the lower semi-continuity of the function $(x, w) \mapsto p_x(w)$ and of the continuity of the function $(x, w) \mapsto F(x, w)$ (\del{see }Lemma~\ref{lm:plsc}\del{ for the details}). This implies that the set defined for all $k \in \Zplusstrict$ and for all $x \in \xdom$ as
\begin{equation}
\Oxk := \lbrace \w \in \R^{kp} | p_x^k(\w) > 0 \rbrace
\end{equation}
is open.
Remark that given that $p_x^k$ is a density, $\Oxk[k]$ is non-empty for all $k \in \Zplusstrict$.

In the case of a Markov chain following \eqref{eq:gmodel} modeled via $\alpha(x,u) = u$, the set $\Oxk$ is the $k$-fold product $\OOpen_{\!\xdom}^k$ where $\OOpen_{\!\xdom}$ is the open control set $\{ x : p(x) > 0\} $. 

The deterministic system
$$
\Sxk[k+1](w_1, \ldots,w_{k+1}) = F( \Sxk[k](w_1,\ldots,w_k) , w_{k+1}),  k \in \Zplusstrict
$$
for $ (w_1, \ldots, w_k,w_{k+1}) $ in the open set $ \Oxk[k+1]$
is the associated deterministic control model CM(F) for Markov chains following \eqref{eq:fmodel}. Given an initial condition $x \in \xdom$, the control model CM(F) is characterized by $F$ and the sets $\Oxk$ for $k \in \Zplusstrict$.
Remark that the control sequence $(w_1,\ldots,w_k) $ lies in the set $\Oxk[k]$ that depends on the initial condition $x$. In contrast when $\alpha(x,u)=u$, the control sequence $(w_1,\ldots,w_k) $ lies in $\Oo_{\xdom}^k$ which is independent of $x$. In analogy to this later case, the sets $\Oxk[k]$ are termed control sets.





\begin{example}[CM(F) for the toy examples]\label{example:CMFtoy}
For the Markov chains defined via \eqref{eq:RW} and \eqref{eq:example} and the densities \eqref{eq:densityToy1} and \eqref{eq:densityToy2}, CM(F) is defined by $F(x,w) = x+w$ and the control sets $\Oxk[k]$ that equal $\R^k$ for all $x \in \R$ and for all $k$.
\end{example}
Hence since both examples share the same $F$ and the same control sets $\Oxk[1]$, we will see later on that proving the $\varphi$-irreducibility or aperiodicity for both Markov chains via the conditions we derive in the paper is the same.

The extended transition map function and the extended probability density can be used to express the transition kernel of $\Phi$ in the following way
\begin{equation}
P^k(x,A ) = \int 1_A (\Sxk[k][x](\w) ) p_x^k(\w) d \w \, ,~ k \in \Zplusstrict, ~A \in \borel(\xdom) .
\end{equation}
This expression will be useful in many proofs.

For a given initial point $x \in \xdom$, we will consider ``deterministic paths" $\w$---where $\w$ is in a control set $\Oxk$---that bring to $A \in \borel(\xdom)$. We give more precisely the following definition of a \emph{k-steps path} from $x$ to $A$.
%

\begin{definition}[$k$-steps path]
For $\x \in \xdom$, $A \in \borel(\xdom)$ and $k \in \Zplusstrict$, we say that $\w \in \R^{kp}$  is a \emph{$k$-steps path} from $\x$ to $A$ if $\w \in \Oxk$ and $\Sxk(\w) \in A$.
\end{definition}

\todo{TA1: Feb 2017: Make a small paragraph that when we talk later on about points that can be reached, it's relating to the control sets, or the fact that we consider possible paths when the w belong to the control set.}
\todo{Todo: check the first time we talk about reaching states and clarify what is meant.}

\subsection{Accessibility, Attracting and Attainable States}
\alex{reachability?}

Following~\cite[Chapter~7]{meyntweedie}, for $x \in \xdom$ and $k \in \N$ we define $A_+^k(x)$, the set of all states that can be reached from $x$ in $k$ steps by CM(F), i.e.\ $A_+^0(x) := \lbrace x \rbrace$ and
$$
A_+^k(x) := \lbrace \Sxk(\w) | \w \in \Oxk \rbrace.
$$
The set of all points that can be reached from $x$ is defined as
$$
A_+(x) := \bigcup_{k \in \N} A_+^k(x) .
$$
If for all $x \in \xdom$, $A_+(x)$ has non empty interior, the deterministic control model CM(F) is said to be \emph{forward accessible}~\cite{controllabilityLieJakubczyk}.

\begin{example}\label{example:forward-accessible}
For the Markov chains \eqref{eq:RW} and \eqref{eq:example} with the common CM(F) model defined in Example~\ref{example:CMFtoy}, $A_+^1(x)=\R$ and thus $A_+(x)=\R$. Therefore CM(F) is forward accessible.
\end{example}

We remind now the definition of a \emph{globally attracting state} \cite[Chapter~7]{meyntweedie}.
A point $\xstar \in \xdom$ is called globally attracting if for all $y\in \xdom$,
\begin{equation} \label{eq:GA}
\xstar \in \Omega_+(y) := \bigcap_{N=1}^{+\infty} \overline{ \bigcup_{k=N}^{+\infty} A_+^k(y) }.
\end{equation}
As we show later on, the existence of a globally attracting state for CM$(F)$ is linked to the $\varphi$-irreducibility of the associated Markov chain. We establish here a basic yet useful proposition giving equivalent statements to the definition of a globally attracting state. This proposition will be heavily used in the different proofs.
\begin{proposition}[Characterization of globally attracting states] \label{pr:GA}
Suppose that $\markov$ follows model \eqref{eq:fmodel} and that conditions $\mathrm{A1}-\mathrm{A4}$ hold.
A point $\x^* \in X$ is a globally attracting state if and only if one of the three following equivalent conditions holds:
\begin{enumerate}
\item[$\mathrm{(i)}$] for all $\y \in \xdom$, $\xstar \in \overline{A_+(\y)}$,
\item[$\mathrm{(ii)}$] for all $\y \in \xdom$ and all open $\Uopen \in \borel(\xdom)$  containing $\xstar$, there exists $k\in \Zplusstrict$  and a $k$-steps path from $\y$ to $\Uopen$,
\item[$\mathrm{(iii)}$] for all $\y \in \xdom$, there exists a sequence $\{ \y_k : k \in \Zplusstrict\}$ with $\y_k \in A_+^k(\y)$ from which a subsequence converging to $\x^*$ can be extracted.
\end{enumerate}
\end{proposition}

We will show that in our context, globally attracting states for CM(F) are equivalent to the notion of reachable states for the associated Markov chain. 
We remind that a point $x \in \xdom$ is called \emph{reachable}~\cite[6.1.2]{meyntweedie} for a Markov chain if for every open set $\OOpen \in \mathscr{B}(\xdom)$ containing $x$ 
$$
\sum_{k=1}^\infty P^k(y,\OOpen) > 0, \,\, ~~ \forall y \in \xdom . \footnote{Another definition of reachability considers the sum from $k$ equal $0$ instead of $1$. Both definitions are equivalent: $\sum_{k \in \Zplusstrict} P^{k}(y, \Oo) = \int_{\xdom} P(y, dz) \sum_{k \in \Zplus} P^k(z, \Oo)$ with Tonelli's theorem. So if $\sum_{k \in \Zplus} P^k(z, \Oo) > 0$ for all $z \in \xdom$, $\sum_{k \in \Zplusstrict} P^{k}(y, \Oo) > 0$. The other implication is trivial, hence the equivalence between the two definitions.}
$$
The equivalence between globally attracting states and reachable states relies on the following proposition.
\begin{proposition} \label{pr:pointproba}
Suppose that $\markov$ follows model \eqref{eq:fmodel}, that conditions $\mathrm{A1}-\mathrm{A4}$ hold and that the function $F$ is continuous. Then for all $\Oo \in \borel(\xdom)$ open set, $x \in \xdom$ and $k \in \Zplusstrict$, the following statements are equivalent:
\begin{itemize}
\item[$\mathrm{(i)}$] there exists a $k$-steps path from $x$ to $\Oo$
\item[$\mathrm{(ii)}$] $P^k(x, \Oo) > 0$.
\end{itemize}
\end{proposition}
We can now deduce the following corollary.
\begin{corollary} \label{cr:GAreachability}
Suppose that $\markov$ follows model \eqref{eq:fmodel}, that conditions $\mathrm{A1}-\mathrm{A4}$ hold and that the function $F$ is continuous. Then $x \in \xdom$ is globally attractive for CM(F) if and only if it is reachable for the associated Markov chain.
\end{corollary}

\alex{We may need to justify why we use the concept of globally attracting rather than reachable, since we showed that they are here equivalent.}
\anne{But globally attracting is defined from the control model and reachable is defined from the Markov chain (using the transition kernel). }

We introduce now two new definitions to characterize some specific states of the underlying deterministic control model. First of all, we have seen that a globally attracting state can be approached via the control model arbitrarily close from any other point in the state space. We introduce the notion of \emph{attainable state} for states that can be visited in finite time from any other point. That is $\xstar \in \xdom$ is an \emph{attainable state} if
\begin{equation} \label{eq:bla}
\xstar \in A_+(y), \textrm{ for all } y \in \xdom . 
\end{equation}
This statement is actually equivalent to
\begin{equation}\label{eq:blabla}
\xstar \in \bigcup_{k\in \Zplusstrict} A_+^k(y), \textrm{ for all } y \in \xdom,
\end{equation}
which turns out to be often more practical to use.\footnote{The equivalence can be easily seen. Indeed suppose that \eqref{eq:bla} holds. Let $\y \in \xdom$ and $u \in \Oxk[1][y]$. Then $\xstar \in A_+(\Sxk[1][y](u)) \subset \bigcup_{k\in \Zplusstrict} A_+^k(y)$. The other implication is immediate.}
Comparing \eqref{eq:bla} and Proposition~\ref{pr:GA}~(i), we see that an attainable state is globally attracting. We will show in Proposition~\ref{pr:GAEA} that the existence of a globally attracting state under some conditions implies the existence of an attainable state.

Second, we introduce the notion of a \emph{steadily attracting state}, whose existence as we will later on show, is linked to the $\varphi$-irreducibility and aperiodicity of the associated Markov chain. We say that a point $\xstar$ is a \emph{steadily attracting state} if for any $\y \in \xdom$ and any open $\Uo \in\borel(\xdom)$ containing $\xstar$, there exists $T \in \Zplusstrict$ such that for all $k \geq T$ there exists a $k$-steps path from $\y$ to $\Uo$. 
In the next proposition we state some quite immediate relations between globally attracting states and steadily attracting ones. Additionally we provide a characterization of a steadily attracting state.
\begin{proposition} \label{pr:SGA}
Suppose that $\markov$ follows model \eqref{eq:fmodel} and that conditions $\mathrm{A1}-\mathrm{A4}$ hold. The following statements hold:\\
$\mathrm{(i)}$ If $\xstar \in \xdom$ is steadily attracting, then it is globally attracting.\\
$\mathrm{(ii)}$ A state $\xstar \in \xdom$ is steadily attracting if and only if for all $\y \in  \xdom$ there exists a sequence $\{ \y_k : k \in \Zplusstrict \}$ with $\y_k \in A_+^k(\y)$, which converges to $\xstar$. \\
$\mathrm{(iii)}$ Assume $F$ is  $C^0$. If there exists a steadily attracting state, then all globally attracting states are steadily attracting.
\end{proposition}

Proposition~\ref{pr:saab} also establishes that under a controllability condition, a globally attracting state $\xstar \in \xdom$ where we can come back in $a$ and $b$ steps with $\gcd(a,b) = 1$ is steadily attracting (that is $\Sxk[a][\xstar](\w_a) = \Sxk[b][\xstar](\w_b) = \xstar$ for some $\w_a \in \Oxk[a][\xstar]$ and $\w_b \in \Oxk[b][\xstar])$; and that if a steadily attracting state exists, then such a globally attracting state necessarily exists.

An attainable state may also be steadily attracting under conditions stated in the following proposition, which will be key to showing that the existence of a steadily attracting state is necessary for the aperiodicity of the Markov chain. 
\begin{proposition} \label{pr:EA2AA}
Suppose that $\markov$ follows model \eqref{eq:fmodel} and that conditions  $\mathrm{A1}-\mathrm{A4}$ hold.
Let $\xstar \in \xdom$ be an attainable state, and consider the set
\begin{equation} \label{eq:E}
E := \lbrace a \in \Zplusstrict | \exists\, t_0 \in \Zplus ,  \forall t \geq t_0,  \xstar \in A_+^{at}(\xstar) \rbrace .
\end{equation}
The following statements hold:

$\mathrm{(i)}$ $E$ is not empty and for all $(a,b) \in E^2$, $\gcd(a,b) \in E$,

$\mathrm{(ii)}$ if $\gcd(E) = 1$, then $\xstar$ is steadily attracting, \alex{What will be useful later is the contrapositive: if $\xstar$ is not steadily attracting, then the gcd of the set is larger than $1$.}

$\mathrm{(iii)}$ if $\markov$ is $\varphi$-irreducible then there exists a $d$-cycle, where $d = \gcd(E)$.
\end{proposition}


\begin{remark}[Control model and choice of density]

For a Markov chain following model~\eqref{eq:fmodel} under conditions $\mathrm{A1}-\mathrm{A5}$, the random variable $\alpha(x, U_1)$ admits different densities which differ on sets of null measure. Therefore there is not a unique deterministic control model associated to a Markov chain following~\eqref{eq:fmodel}. The control sets $\{ \Oxk : k \in \Zplusstrict, x\in \xdom\}$ and the sets $\{A_+^k(x): k \in \Zplusstrict , x\in \xdom\}$ may differ depending on the choice of the density for representing the Markov chain. 

However, under conditions $\mathrm{A1-A4}$ and that $F$ is $C^0$, the choice of a different lower semi-continuous density does not affect whether a point $\xstar \in \xdom$ is globally or steadily attracting\footnote{Indeed, if $\xstar$ is globally attracting then for all $y \in \xdom$ and $\Uo_{\xstar}$ open neighborhood of $\xstar$ there exists $\w$ a $k$-steps path from $y$ to $\Uo_{\xstar}$ for some $k \in \Zplusstrict$; since $\pxk[k][y]$ is lower semi-continuous, for any other density $\tilde{p}_y^k$ equal almost everywhere to $\pxk[k][y]$, for any $\eta > 0$ there exists a $\u \in B(\w, \eta)$ such that $\tilde{p}_y^k(\u) > 0$. By continuity of $\Sxk[k][y]$, we may then have a $\u$ such that $\Sxk[k][y](\u) \in \Uo_{\xstar}$ and $\tilde{p}_y^k(\u) > 0$, which implies that $\xstar$ is a globally attracting state for $(\tilde{p}_y)_{y \in \xdom}$. The same reasoning applies for steadily attracting states.}.  Note that while a globally (resp. steadily) attracting state for a lower semi-continuous density is therefore globally (resp. steadily) attracting for any other density representing the same random variable (even non-lower semi-continuous densities), the converse does not hold in general.

Attainable states may depend on the choice of density. However, under conditions $\mathrm{A1-A5}$ and if there exists $\xstar \in \xdom$ a globally attracting state,  $k \in \Zplusstrict$ and $\wstar \in \Oxk[k][\xstar]$ such that the controllability matrix $\Cxk[k][\xstar](\wstar)$ has rank $n$, then as a consequence of Proposition~\ref{pr:GAEA} the existence of an attainable state is independent of the choice of density.
\alex{I refer to densities, but it would be more accurate to talk about families of densities}
\end{remark}

\begin{remark}
In Proposition~\ref{pr:saab} in the appendix, we show that under a controllability condition, the existence of a steadily attracting state is equivalent to the existence of a globally attracting state $\xstar$ for which there exists paths of length $a$ and $b$ with $\gcd(a,b) = 1$ leading from $\xstar$ to $\xstar$. This result is particularly useful to obtain practical conditions to prove that a globally attracting state is steadily attracting as stated in Lemma~\ref{lem:PCSA}.
\end{remark}
\alex{I put it in a remark environment and made the text less technical.}

\subsection{Controllability Matrix and Controllability Condition}

A central condition in many of our results is that the rank of the so-called controllability matrix is $n$. This condition \new{is a straightforward generalization of} \del{generalizes} the controllability condition for linear state-space models \cite{meyntweedie}. In this section we give some background on controllability matrices and derive some first results related to the rank condition above. Our notations are borrowed from \cite{meyn-caines1991,meyntweedie}.

\subsubsection{Controllability Matrix: Definition and First Properties}

For an initial condition $y \in \xdom$ and a sequence $\lbrace w_k \in \R^{p} : k \in \N \rbrace$, let $\lbrace A_k, B_k : k \in \N \rbrace$ denote the matrices
\begin{align}\label{eq:Ak}
A_k &= A_k(y, w_1, \ldots, w_{k+1}) := \left[ \frac{\partial F}{\partial x} \right]_{(\Sxk[k][y], w_{k+1})} \\\label{eq:Bk}
B_k &= B_k(y, w_1, \ldots, w_{k+1}) := \left[ \frac{\partial F}{\partial w} \right]_{(\Sxk[k][y], w_{k+1})}
\end{align}
and let $\Cxk[k][y] = \Cxk[k][y](w_1, \ldots, w_k) \in \R^{n \times pk}$ denote the \emph{generalized controllability matrix} (along the sequence $(w_1, \ldots, w_k)$)
\begin{equation}
\Cxk[k][y](w_1, \ldots, w_k) := \left[ A_{k-1} \ldots A_1 B_0 | \ldots | A_{k-1} B_{k-2} | B_{k-1} \right].
\end{equation}
Remark that from \eqref{eq:Ak} and \eqref{eq:Bk}, it follows immediatly that for $k \in \Zplusstrict$
\begin{align*}
B_{k}(y,w_1,\ldots,w_{k+1}) &= B_{k-1}(S_y^1(w_1),w_2,\ldots,w_{k+1})\\
A_{k}(y,w_1,\ldots,w_{k+1}) &= A_{k-1}(S_y^1(w_1),w_2,\ldots,w_{k+1}) ,
\end{align*}
and therefore, the controllability matrix satisfies for $k \in \Zplusstrict$
\begin{equation} \label{eq:controlrec}
\Cxk[k][y](w_1, \ldots, w_{k}) = \left[ \left. A_{k-1} \ldots A_1 B_0 \right| \Cxk[k-1][{\Sxk[1][y](w_1)}](w_2, \ldots, w_{k}) \right] .\\ 
\end{equation}
Inductively, it follows that for $i = 1, \ldots , k$, $\Cxk[k-i][{\Sxk[i][y](w_1, \ldots, w_i)}](w_{i+1}, \ldots, w_k)$ is a sub-matrix of $\Cxk[k][y](w_1, \ldots, w_k)$.
Additionally, the generalized controllability matrix $\Cxk[k][y]$ is the Jacobian matrix of the function $(w_1, \ldots, w_k) \mapsto \Sxk[k][y](w_1, \ldots, w_k)$, that is for $\w_0 \in \R^{kp}$
\begin{equation}\label{eq:Cpartial}
\Cxk[k][y](\w_0) = \left[ \frac{\partial S_y^k}{\partial w_1}  \left|\right. \ldots \left|\right.  \frac{\partial S_y^k}{\partial w_{k}} \right]_{\w_0}.
\end{equation}
This formula is a consequence of the chain rule. We provide in appendix its derivation. 
A central condition for our results will be that the rank of the controllability matrix $\Cxk(\w)$ is $n$  for some $x \in \xdom$, $k \in \Zplusstrict$ and $\w \in \Oxk$. This condition is equivalent to $\Sxk$ being a submersion at $\w$. This formulation was used in a previous version of this work~\cite{chotard2015arxiv}.

\todo{we make this abuse of notations that we take the derivative wrt a variable and apply it to its variable - check what is standard in statistics litterature. Remark: MT in 7.4 does not make this abuse of notation (maybe we should not do it when it is "important" or in definition.}

\subsubsection{Accessibility and Controllability Condition}
We show in the next proposition that under conditions $\mathrm{A1}-\mathrm{A5}$, if the rank condition on the controllability matrix is satisfied at a globally attracting state, it is satisfied at any $x$ in $\xdom$.
\begin{proposition} \label{pr:gacontrol}
Suppose that $\markov$ follows model \eqref{eq:fmodel} and that conditions $\mathrm{A1}-\mathrm{A5}$ hold.
Let $\xstar \in \xdom$ be a globally attracting state. If there exists $k \in \Zplusstrict$ and $\wstar \in \Oxk[k][\xstar]$ such that $\rank \Cxk[k][\xstar](\wstar) = n$, then for all $x \in \xdom$, there exists $T \in \Zplusstrict$ and $\uu \in \Oxk[T][x]$ for which $\rank(\Cxk[T][x](\uu)) = n$.
\end{proposition}

In the next proposition, we show that from a globally attracting state where the rank condition on the controllability matrix is satisfied, we can construct an attainable state. This proposition will later on allow us to use Proposition~\ref{pr:EA2AA} to prove that if for all $x \in \xdom$ the rank condition is satisfied, then the existence of a steadily attracting state is a necessary condition for the aperiodicity of the Markov chain.
\begin{proposition} \label{pr:GAEA}
Suppose that $\markov$ follows model \eqref{eq:fmodel} and that conditions  $\mathrm{A1}-\mathrm{A5}$ hold. Let $\xstar \in \xdom$ and suppose there exists $k \in \Zplusstrict$ and $\wstar \in \Oxk[k][\xstar]$ for which $\rank \Cxk[k][\xstar](\wstar) = n$. 
\begin{itemize}
	\item[(i)] There exists $\Uo$ a neighborhood of $\xstar$ such that for all $x \in \Uo$, there exists $\w \in \Oxk$ for which $\Sxk(\w) = \Sxk[k][\xstar](\wstar)$.
	\item[(ii)] If $\xstar$ is globally attracting, then $\Sxk[k][\xstar](\wstar)$ is attainable.
\end{itemize}

\del{
Let $\xstar \in \xdom$ be a globally attracting state and suppose there exists $k \in \Zplusstrict$ and $\wstar \in \Oxk[k][\xstar]$ for which $\rank(\Cxk[k][\xstar](\wstar)) = n$. Then $\Sxk[k][\xstar](\wstar)$ is an attainable state.}
\alex{We could add that there exists a neighborhood of $\xstar$ that allows all points inside to go in $k$-steps to $\Sxk[k][\xstar](\wstar)$; which is shown in the proof and that we use in other propositions. ++DONE++}


\end{proposition}

When the Markov chain reduces to \eqref{eq:gmodel}, that is $\alpha(x,u)=u$, and $F$ is $C^\infty$, the rank condition holding for every $x$ is equivalent to forward accessibility~\cite[Proposition~2.3]{controllabilityLieJakubczyk}. This results relies on the inverse function theorem and Sard's theorem and can be easily generalized \new{to}\del{as is shown in} the following proposition.

\begin{proposition} \label{pr:forwardaccess}
Suppose that $\markov$ follows model \eqref{eq:fmodel} and that conditions  $\mathrm{A1}-\mathrm{A5}$ hold. If for all $x \in \xdom$ there exists $k \in \Zplusstrict$ and $\w \in \Oxk$ such that $\rank(\Cxk(\w)) = n$, then $\CM(F)$ is forward accessible.

Furthermore, if $F$ is $C^\infty$, $\CM(F)$ is forward accessible if and only if for all $x \in \xdom$ there exists $k \in \Zplusstrict$ and $\w \in \Oxk$ for which $\rank(\Cxk(\w)) = n$.
\end{proposition}
Remark that the fact that the rank condition on the controllability matrix implies forward accessibility still holds for $F$ a $C^1$ function was already noted in \cite{controllabilityLieJakubczyk}.

\section{Main Results} \label{sec:main}

\subsection{T-chain and Irreducibility}

In this section, we state our main results on the $\varphi$-irreducibility and T-chain property. 
On the one hand, we generalize the result holding for a Markov chain following \eqref{eq:gmodel} with $F$ being $C^\infty$ that if CM(F) is forward accessible, the associated Markov chain is $\varphi$-irreducible if and only if CM(F) admits a globally attracting state \cite[Proposition~7.2.6]{meyntweedie}. We prove more precisely that under the conditions $\mathrm{A1}-\mathrm{A5}$,  if for all $x$, there exists $k$ and $\wstar \in \Oxk[k][\xstar]$ such that $\rank (\Cxk[k][\xstar](\wstar)) = n$ (we have seen that this condition implies forward accessibility), then the $\varphi$-irreducibility of a chain following \eqref{eq:fmodel} is equivalent to the existence of a globally attracting state.

We then derive a practical condition by showing that the existence of a globally attracting state where the rank condition is satisfied implies that the associated Markov chain is a $\varphi$-irreducible T-chain and thus that every compact set is petite. That is, we only need to find a globally attracting state and verify the rank condition at this state to prove the $\varphi$-irreducibility and T-chain property.

%
%

Those results rely on the generalization to our context of \cite[Theorem~2.1 (iii)]{meyn-caines1991}. In particular we show that around a point $x \in \xdom$ where the rank of the controllability matrix $\Cxk(\w)$ is $n$ for some $k \in \Zplusstrict$ and $\w \in \Oxk$, there exists an open small set containing $x$. More precisely we have the following result.

\begin{proposition} \label{pr:2.1}
Suppose that $\markov$ follows model \eqref{eq:fmodel} and that conditions $\mathrm{A1}-\mathrm{A5}$ are satisfied. 

$\mathrm{(i)}$ Let $x \in \xdom$, if $\rank \Cxk(\w) = n$ for some $k \in \Zplusstrict$ and $\w \in \Oxk$, then there exists $c>0$, and open sets $\Uo_x$ and $\Vo_x^\w$ containing $x$ and $\Sxk(\w)$, respectively,  such that 
\begin{equation} \label{eq:2.1}
P^k(y, A) \geq c \mleb(A \cap \Vo_x^\w),  \textrm{ for all } y \in \Uo_x ,  A \in \borel(\xdom).
\end{equation}
That is, $\Uo_x$ is a $\nu_k$-small set where $\nu_k : A \mapsto c \mleb(A \cap \Vo_x^\w)$. 
\anne{This is (iii) of Theorem 2.1 of Meyn Caines.}

$\mathrm{(ii)}$ If furthermore $F$ is $C^\infty$, and if for some $x \in \xdom$, there exists $k \in \Zplusstrict$, $c >0$ and $\Vo$ an open set such that 
\begin{equation} \label{eq:wsc}
P^k(x, A) \geq c \mleb(A \cap \Vo), \textrm{ for all } A \in \borel(\xdom),
\end{equation}
then $\rank \Cxk(\w)=n$ for some $\w \in \Oxk$, a $k$-steps path from $x$ to $\Vo$.
\anne{Connexion to Theorem 2.1 to be discussed.}
\end{proposition}
The proof of this proposition is very similar to the proof of \cite[Theorem~2.1]{meyn-caines1991}. We present it in the appendix where we additionally highlight \todo{is it true?}\alex{yep} the differences to the proof of \cite[Theorem~2.1]{meyn-caines1991}. 
Note that if $F$ is $C^\infty$ and \eqref{eq:wsc} holds for some $x \in \xdom$, then by Proposition~\ref{pr:2.1} $\mathrm{(i)}$ and $\mathrm{(ii)}$,  \eqref{eq:wsc} holds for all $y$ in an open neighborhood of $x$ and $\Vo$ a possibly smaller open set.

We deduce from Proposition~\ref{pr:2.1} that if for all $x \in \xdom$ there exists $k \in \Zplusstrict$ and $\w \in \Oxk$ for which $\rank(\Cxk(\w)) = n$, the state-space may be written as the union of open small sets and hence $\markov$ is a $T$-chain (see \cite[Proposition~6.2.3, Proposition~6.2.4]{meyntweedie}). This result is formalized in the following corollary, which can be seen as a generalization \new{to our model} of \cite[Proposition~7.1.5]{meyntweedie}.

\begin{corollary} \label{cr:controlTchain}
Suppose that $\markov$ follows model \eqref{eq:fmodel} and that conditions $\mathrm{A1}-\mathrm{A5}$ are satisfied. 
Suppose that for all $x \in \xdom$, $\rank( \Cxk(\w)) = n$ for some $k \in \Zplusstrict$ and $\w \in \Oxk$. Then $\xdom$ can be written as the union of open small sets and thus $\markov$ is a $T$-chain.
\end{corollary}
\anne{I added again the union of open small sets as this is the key idea (given by the reviewer and that I will also personnally remember)}

\begin{proof}
Under the conditions of the corollary, the conditions of Proposition~\ref{pr:2.1} $(\mathrm{i})$ are satisfied for all $x \in \xdom$. Hence, for all $x \in \xdom$, there exists $\Uo_x$ an open $\nu_k$-small set containing $x$.
A $\nu_k$-small set is a $\nu_a$-petite set (with $a$ a Dirac distribution at $k$), hence, according to \cite[Proposition~6.2.3]{meyntweedie}, $K_a$ possesses a continuous component $T_x$ non trivial on all of $\Uo_x$, and so in particular non trivial at $x$. Hence, according to \cite[Proposition~6.2.4]{meyntweedie} $\markov$ is a $T$-chain.
\end{proof}

To prove the equivalence between  $\varphi$-irreducibility and the existence of a globally attracting state, we \del{need to}\new{first} \alex{we don't really need to} characterize the support of the maximal irreducibility measure in terms of globally attracting states. More precisely the following holds.
\begin{proposition} \label{pr:support}
Suppose that $\markov$ is a $\psi$-irreducible Markov chain following model \eqref{eq:fmodel}, with $\psi$ the maximal irreducibility measure, that $\mathrm{A1}-\mathrm{A4}$ hold and that $F$ is $C^0$. Then
\begin{equation} \label{eq:suppGA}
\supp \psi = \lbrace \xstar \in \xdom | \, \xstar \textrm{ is globally attracting} \rbrace .
\end{equation}
Furthermore, let $\xstar \in \xdom$ be globally attracting, then
\begin{equation} \label{eq:suppAxstar}
\supp \psi = \overline{A_+(\xstar)} .
\end{equation}
\end{proposition}
We are now ready to state our result generalizing \cite[Proposition~7.2.6]{meyntweedie} \new{to our model}.
\begin{theorem} \label{th:irr}
Suppose that $\markov$ follows model \eqref{eq:fmodel} and that conditions $\mathrm{A1}-\mathrm{A5}$ are satisfied. Suppose that for all $x \in \xdom$, there exists $k \in \Zplusstrict$ and $\w \in \Oxk$ such that $\rank \Cxk(\w) = n$. Then $\markov$ is $\varphi$-irreducible if and only if a globally attracting state exists.
\end{theorem}
\begin{proof}
Suppose that for all $x \in \xdom$ there exists $k \in \Zplusstrict$ and $\w \in \Oxk$ such that $\rank \Cxk(\w) = n$. If $\markov$ is a $\varphi$-irreducible chain, then by Proposition~\ref{pr:support} any point of its support is globally attracting. Since $\varphi$ is not trivial, its support is not empty and so there exists a globally attracting state. 

Conversely, suppose that there exists $\xstar$ a globally attracting point, which is also reachable by Corollary~\ref{cr:GAreachability}. By hypothesis, according to Corollary~\ref{cr:controlTchain}, $\markov$ is a $T$-chain and so by \cite[Proposition~6.2.1]{meyntweedie},\mathnote{Prop 6.2.1: A T-chain that contains one reachable point is $\varphi$-irreducible.} $\markov$ is $\varphi$-irreducible, which concludes the proof.
\end{proof}

As discussed above, in the particular case of $\alpha(x,u)=u$ and $F$ is $C^\infty$, the similar result derived in \cite[Proposition~7.2.6]{meyntweedie}  states that if CM(F) is forward accessible, the associated Markov chain is $\varphi$-irreducible if and only if CM(F) admits a globally attracting state. In the more general context of bounded positive kernels, the existence of an open reachable\footnote{A set $A \in \borel(\xdom)$ is said reachable if for all $x \in \xdom$ there exists some $k \in \Zplusstrict$ for which $P^k(x, A) > 0$.} petite set is equivalent to the Markov chain being a $\varphi$-irreducible $T$-chain for which compact sets are petite~\cite[Theorem~2.4]{cline2011irreducibility}.


From Proposition~\ref{pr:gacontrol}, we  know that if the rank condition on the controllability matrix is satisfied at a globally attracting state, it is satisfied for all $x \in \xdom$. Hence, we can deduce the practical condition that if there exists a globally attracting state where the rank condition on the controllability matrix is satisfied, then the associated Markov chain is a $\varphi$-irreducible, T-chain and thus every compact set is petite.

%

\begin{theorem}[Practical Condition for $\varphi$-irreducibility]\label{th:irr-practical}
Suppose that $\markov$ follows model \eqref{eq:fmodel} and that conditions $\mathrm{A1}-\mathrm{A5}$ are satisfied.
If there exists $\xstar$ a globally attracting state, and if $\rank \Cxk[k][\xstar](\wstar) = n$ for some $k \in \Zplusstrict$ and $\wstar \in \Oxk[k][\xstar]$, then $\markov$ is a $\varphi$-irreducible $T$-chain, and thus every compact set is petite. \alex{we can be more precise and say that there exists an open neighborhood of $\Sxk[k][\xstar](\wstar)$ where the irreducibility measure dominates the Lebesgue measure. Unrelatedly, we get the every compact set is petite for free.}
\end{theorem}
\begin{proof}
Suppose there exists $\xstar$ a globally attracting state, $k \in \Zplusstrict$ and $\wstar \in \Oxk[k][\xstar]$ such that $\rank \Cxk[k][\xstar](\wstar) = n$. By Proposition~\ref{pr:gacontrol} for all $x \in \xdom$ there exists $t \in \Zplusstrict$ and $\w \in \Oxk[t]$ for which $\rank \Cxk[t](\w) = n$. According to Corollary~\ref{cr:controlTchain}, $\markov$ is a T-chain and according to Theorem~\ref{th:irr}, $\markov$ is $\varphi$-irreducible. Hence according to \cite[Theorem~6.2.5]{meyntweedie} every compact set is petite.
\mathnote{Another way to prove it using Cline's paper: with Proposition~\ref{pr:2.1}, there exists $\Uo_{\xstar}$ an open small set containing $\xstar$. Since $\xstar$ is globally attracting, by Proposition~\ref{pr:GA} there exists $t \in \Zplusstrict$ and $\u \in \Oxk[t][\xstar]$ a $t$-steps path from $\xstar$ to $\Uo_{\xstar}$, and so by Proposition~\ref{pr:pointproba}, $P^t(\xstar, \Uo_{\xstar}) > 0$. Therefore by \cite[Theorem~2.2]{cline2011irreducibility} (and with the fact that $\xstar$ is reachable with Corollary~\ref{cr:GAreachability}), $\markov$ is $\varphi$-irreducible.}

\end{proof}
%

\subsection{Aperiodicity}
We show now that the results of the previous section can be transposed to prove $\varphi$-irreducibility \emph{and} aperiodicity of a Markov chain if we replace the condition of the existence of a globally attracting state  by the existence of a steadily attracting state. 

We first state the equivalence between the existence of a steadily attracting state and the $\varphi$-irreducibility and aperiodicity of the associated Markov chain.
\begin{theorem} \label{th:ap}
Consider a Markov chain $\markov$ following the model \eqref{eq:fmodel} for which conditions  $\mathrm{A1}-\mathrm{A5}$ are satisfied.  If for all $x \in \xdom$, there exists $k \in \Zplusstrict$ and $\w \in \Oxk$ for which $\rank(\Cxk(\w)) = n$, then $\markov$ is a $\varphi$-irreducible aperiodic Markov chain if and only if there exists a steadily attracting state.
\end{theorem}
A related result proving a necessary and sufficient condition for a chain to be aperiodic has been derived for Markov chains following \eqref{eq:gmodel} under the assumption that the control set $\Oo_\xdom= \lbrace u \in \R^p \ | p(u) > 0 \rbrace$  is connected. More precisely, it has been shown that if there exists $\xstar$ a globally attracting state, then $\markov$ is aperiodic if and only if $\overline{A_+(\xstar)}$ is connected ~\cite[Proposition~7.2.5, 7.3.4, Theorem~7.3.5]{meyntweedie}.
Note that the condition that $\Oo_\xdom$ is connected is critical: if a Markov chain has a non-connected set $\Oo_\xdom$, then the equivalence that $\overline{A_+(\xstar)}$ is connected if and only if the Markov chain is aperiodic does not hold anymore.
A trivial (albeit artificial) example is to take an i.i.d. sequence of random variables $\{ \Uk : k \in \Zplusstrict \}$ with non connected support $(-2, -1) \cup (1,2)$, and to consider it as our Markov chain of interest $\pk = \U_k$. Then $A_+^k(x) = (-2, -1) \cup (1,2)$ for all $k \in \Zplusstrict$, and so $\overline{A_+(x)}$ is not connected for any $x \in \R$ but the Markov chain is aperiodic.

In the general context of a bounded positive weak-Feller transition kernel $Q$, if there exists $B$ a $\nu_a$-small set (i.e. $Q^a(x, A) \geq \nu_a(A)$ for all $x \in B$ and $A \in \borel(\xdom)$) and neighborhood of a reachable point $\xstar$ such that $Q^b(\xstar, B) > 0$ for some $b \in \Zplusstrict$ for which $\gcd(a,b) = 1$, then $Q$ is a $\varphi$-irreducible aperiodic $T$-chain; and a slightly weaker form of the converse holds~\cite[Theorem~2,6]{cline2011irreducibility}. In our more limited context, this can be shown using Theorem~\ref{th:ap} and Proposition~\ref{pr:saab}.

Similarly to Theorem~\ref{th:irr-practical}, we now deduce the following practical condition to prove the aperiodicity of a Markov chain.

\begin{theorem}[Practical condition for $\varphi$-irreducibility and aperiodicity]\label{th:ap-practical}
Consider a Markov chain $\markov$ following the model \eqref{eq:fmodel} for which conditions  $\mathrm{A1}-\mathrm{A5}$ are satisfied. If there exists $\xstar \in \xdom$ a steadily attracting state, $k \in \Zplusstrict$ and $\wstar \in \Oxk[k][\xstar]$ such that $\rank(\Cxk[k][\xstar](\wstar)) = n$, then $\markov$ is an aperiodic $\varphi$-irreducible $T$-chain, and every compact set is small.
\end{theorem}
We give here an outline of how the existence of steadily attracting state implies the aperiodicity of the Markov chain while the full proofs the theorems can be found in the appendix.

Proposition~\ref{pr:2.1} $\mathrm{(i)}$ allows to construct a set $\Uo$ which is a neighborhood of $\xstar$ and a non-trivial measure $\mu_\Vo$ such that if a point $y \in \xdom$ can reach $\Uo$ with a $t$-steps path, then the kernel $P^{t\new{+k}}(y, \cdot)$ dominates the measure $\mu_\Vo$. Since $\xstar$ is steadily attractive, $\Uo$ can be reached from any point $y \in \xdom$ for all time $t \geq t_y$. Hence $P^{t}(y, \cdot)$ dominates $\mu_\Vo$ for all $t \geq t_y$, which implies that $\sum_{t \in \Zplusstrict} P^t(y, \cdot)$ dominates $\mu_\Vo$ for all $y \in \xdom$ and so that $\mu_\Vo$ is an irreducibility measure.

Now consider $(D_i)_{i = 1, \ldots, d}$ a $d$-cycle. Since $\mu_\Vo$ is an irreducibility measure, there exists $D_i$ such that $\mu_\Vo(D_i) > 0$, and with the fact that $P^{t\new{+k}}(y, \cdot)$ dominates $\mu_{\Vo}$ for all $t \geq t_y$, we deduce that $D_i$ can be reached with positive probability from all $y \in \xdom$ in $t$ steps with any $t \geq t_y$. 
By definition of the $d$-cycle, the Markov chain steps with probability one from a set of the cycle to the next, and so for any $m \in \Zplus$ the Markov chain goes in $md + 1$ steps from $D_i$ to $D_{i+1}$ with probability one. This contradicts that $D_i$ can be reached from anywhere (including itself) in $md + 1$ steps for $m$ large enough, unless $d = 1$ meaning the Markov chain is aperiodic.



Theorem~\ref{th:ap-practical} is a generalization of Proposition~3.2 in \cite{meyncaines1989} where a Markov chain following \eqref{eq:gmodel} with $G$ being $C^\infty$ is $\varphi$-irreducible and aperiodic if the control model is forward accessible and \emph{asymptotically controllable}\alex{I added an emphasis on asymptotically controllable, since we're defining it}, that is if there exists $\xstar \in \xdom$ such that for all $y \in \xdom$ there exists a sequence $\{ w_k : k \in \Zplusstrict\}$ with $w_k  \in \Oo_\xdom$ such that the sequence $\{ \Sxk(w_1, \ldots, w_k):  k \in \Zplusstrict \}$ converges to $\xstar$. This latter condition of asymptotic controllability implies the existence of a steadily attracting state (as can easily be seen from Proposition~\ref{pr:SGA}), and is in fact quite stronger: for $\xstar \in \xdom$ a steadily attracting state and any $y \in \xdom$ there exists a sequence $\{ \w_k :  k \in \Zplusstrict \}$ with $\w_k \in \Oo_\xdom^k$ such that $\{ \Sxk(\w_k):  k \in \Zplusstrict \}$ converges to $\xstar$, while in the context of asymptotic controllability $\w_{k+1}$ would be restricted to $(\w_k, u)$ with $u \in \Oo_\xdom$. Remark that hence asymptotic controllability forces for any $\epsilon >0$ the existence of $y \in \xdom$ and $u \in \Oo_\xdom$ such that $\| \Sxk[1][y](u) - y \| \leq \epsilon$.\footnote{Indeed for $t \in \Zplusstrict$ large enough, the convergence of $\Sxk[k][y](\w_k)$ to $\xstar$ forces $\Sxk[t][y](\w_k)$ and $\Sxk[t+1][y](\w_k, u)$ to be contained within the same ball of radius $\epsilon/2$, and so $x := \Sxk[t][y](\w_k)$ is at distance at most $\epsilon$ from $\Sxk[1][x](u)$.} 
This is a restriction on the possible models that can be considered which is not imposed if we consider the condition of the existence of a steadily attracting state. Indeed take an additive random walk on $\R$ with increment distribution with a support of the type $(-\infty,a) \cup (a,+\infty)$ with $a > 0$. It is not difficult to prove that the associated control model \emph{is not asymptotically controllable} while every $x \in \R$ is steadily attractive and the chain is $\varphi$-irreducible.

\mathnote{The condition given equation (13) of the Meyn Caines paper "Asymptotic behavior of ..." p 542 corresponds actually to the condition of global attractivity. They probably realized that and afterwards wrote down Chapter 7 of Meyn Tweedie with globally attractive.}

\section{Applications}\label{sec:appli}

In this section, we illustrate how to use the different conditions we have derived to prove $\varphi$-irreducibility, aperiodicity and that compact are small sets for three examples of Markov chains. We consider first the toy examples presented in Section~\ref{sec:intro} and then turn to a more complex example where it would be very intricate---if not impossible---to prove $\varphi$-irreducibility and aperiodicity by hand. Before to tackle those examples, we summarize the methodology ensuing the results we have developed.

\paragraph{Methodology} The following steps need to be followed to apply Theorem~\ref{th:irr-practical} to prove that a Markov chain is a $\varphi$-irreducible T-chain (and thus that compact sets are petite).
\begin{enumerate}
\item[(i)] Identify that the Markov chain follows model \eqref{eq:fmodel}: exhibit the function $F$ and \new{show the existence of $\alpha$ and $\{ \Uk : k \geq 1 \}$ an i.i.d. sequence such that \eqref{eq:fmodel} holds}\del{ either identify explicitly $\alpha$ or argue on its existence}\alex{juste une proposition}
\item[(ii)]  Identify the density $p_x()$ of $\alpha(z,U_1)$
\item[(iii)] Show that conditions $\mathrm{A1}$ to $\mathrm{A5}$ are satisfied. Particularly prove that $F$ is $C^1$ and $(x,w) \mapsto p_x(w)$ is lower semi-continuous
\item[(iv)] Prove that there exists a globally attracting state $\xstar$
\item[(v)] Show that there exists $k$ and $\w \in \Oxk$ such that $\rank \Cxk[k][\xstar](\wstar) = n $
\end{enumerate}
Similarly, to apply Theorem~\ref{th:ap-practical} to prove that the chain is a $\varphi$-irreducible aperiodic T-chain and that compact are small sets, we need to replace step (iv) above by proving that there exists a steadily attracting state. We highlight in Lemma~\ref{lem:PCSA} two practical conditions to facilitate the proof of existence of a steadily attracting state. Note also that to prove the existence of a globally (or steadily) attracting state, it is practical to have identified the control sets $\Oxk$, yet it might be enough to only know $\Oxk[1][x]$ (see the proof of Proposition~\ref{prop:SA-xNES}).
%

\subsection{Toy Examples}
We consider the two examples introduced in Section~\ref{sec:intro}. For those Markov chains, $\varphi$-irreducibility can be proven directly. Indeed, it follows from the expression of the transition kernel in \eqref{eq:trans-kernel} that for all $A$ with strictly positive Lebesgue measure, for all $x$, $P(x,A)>0$. It is also relatively straightforward to prove aperiodicity and that compact are small sets by minoring \eqref{eq:trans-kernel} for all $x$ in a compact $C$ by $\int_A \min_{x \in C} p_x(x-y) dy $.

Yet in order to illustrate how to use the conditions derived in the paper, we show how Theorem~\ref{th:ap-practical} can be applied to show $\varphi$-irreducibility, aperiodicity and that every compact set is small. 
The function $F$ associated to the chains defined in \eqref{eq:RW} and \eqref{eq:example} equals $F(x,w) = x + w$ and it is thus $C^\infty$. Additionally, we have seen that both chains share the same control model CM(F) which is additionally forward accessible (see Example~\ref{example:forward-accessible}). Thus according to Proposition~\ref{pr:forwardaccess}, the rank condition on the controllability matrix is satisfied for all $x \in \R$. Hence in order to show the $\varphi$-irreducibility, aperiodicity, and that compact sets are small sets, according to Theorem~\ref{th:ap} or Theorem~\ref{th:ap-practical}, it remains to prove the existence of a steadily attracting state.
%
We actually prove in the next proposition that \emph{every} $x$ in $\R$ is steadily attracting.

\begin{proposition} Consider the control model CM(F) defined in Example~\ref{example:CMFtoy} associated to the Markov chains defined in \eqref{eq:RW} and \eqref{eq:example}. Then every $x \in \R$ is steadily attracting for CM(F).
\end{proposition}
\begin{proof}
Since $\Oxk = \R^k$ and $S_y^k(\w) = y+w_1 + \ldots + w_k$, for all $x$ and all initial condition $y$, for all $k \geq 1$, the vector $\bar{\w}= (0, \ldots, 0, y-x) \in \R^k =  \Oxk[k][y]$ satisfies $S_y^k(\bar{\w})=x$. This implies that $x$ is a steadily attracting state. Remark that we have shown a stronger condition than steady attractivity because we have shown that we can exactly hit $x$ in $k$ steps for any $k \geq 1$.
\end{proof}
Consequently, the assumptions of Theorem~\ref{th:ap} are satisfied for the Markov chains defined in \eqref{eq:RW} and \eqref{eq:example} and thus the chains are $\varphi$-irreducible, aperiodic T-chains and every compact set is small. According to Proposition~\ref{pr:support}, we also know that the support of the maximal irreducibility measure equals $\R$.

We have illustrated how the tools developed in the paper unify the study of $\varphi$-irreducibility, aperiodicity and the identification that compact are small sets for the two Markov chains  \eqref{eq:RW} and \eqref{eq:example} whose model defined via \eqref{eq:gmodel} is sensibly different with an associated $G$ which is $C^\infty$ for \eqref{eq:RW} and discontinuous for \eqref{eq:example}.

\subsection{A Step-size Adaptive Randomized Algorithm Optimizing Scaling-invariant Functions}
We consider now a Markov chain stemming from an adaptive stochastic algorithm aiming at optimizing continuous optimization problems.
Proving the stability of the chain is important because it implies the linear convergence (or divergence) of the underlying optimization algorithm which is generally difficult to prove for this type of algorithms.
 This example is not artificial: showing the $\varphi$-irreducibility, aperiodicity and that compact sets are small sets by hand without the results of the current paper seems to be very arduous and actually motivated the development of the theory of this paper. 



\newcommand{\Ytil}{Y_t^{i:\lambda}}
\newcommand{\Ntil}{N_t^{i:\lambda}}
\newcommand{\xNES}{\ensuremath{\mathrm{xNES}}}
\newcommand{\dsa}{\ensuremath{\mathrm{CSAw/o}}}
\newcommand{\dsas}{\ensuremath{{\mathrm{CSAw/o}^{2}}}}
\newcommand{\G}{F}
\newcommand{\LRm}{\kappa_{m}}
\newcommand{\LRsigma}{\kappa_{\sigma}}
\newcommand{\mueff}{\mu_{\rm w}}
\renewcommand{\dim}{n}

We consider a step-size adaptive stochastic search algorithm optimizing an objective function $f: \R^n \to \R$. The algorithm pertains to the class of so-called \emph{Evolution Strategies} (ES) algorithms \cite{Rechenberg} 
that date back to the 70's. The algorithm is however related to information 
geometry: It was recently derived from taking the natural gradient of a joint objective function defined on the Riemannian manifold formed by the family of Gaussian distributions \cite{Glasmachers2010,ollivier2013information}. More precisely, let $X_0 \in \R^n$ and let $\{ \U_k:  k \in \Zplusstrict\}$ be an i.i.d. sequence of random vectors where each $\U_k$ is composed of $\lambda \in \Zplusstrict$ components $\U_k =( \U_k^1,\ldots,\U_k^\lambda) \in \R^{n \lambda}$ with $\{ \U_k^i:  i=1, \ldots \lambda\}$ i.i.d. and following each a standard multivariate normal distribution $\mathcal{N}(0,\Idn)$ where $\Idn$ denotes the identity matrix of size $n $.
Given $(\Xt,\sigma_k) \in \R^n \times \R_{>}$ the current state of the algorithm, $\lambda$ candidate solutions centered on $\Xt$ are sampled using the vector $\Utt$,i.e.\ for $i=1,\ldots,\lambda$
\begin{equation}
\Xt + \st \Utt^i ,
\end{equation}
where $\st$ called the step-size of the algorithm corresponds to the overall standard deviation of $\st \Utt^i$.
Those solutions are ranked according to their $f$-values. More precisely, let $\mathcal{S}$ be the permutation of $\lambda$ elements such that
\begin{equation}\label{eq:selection}
f \left(\Xt + \st \Utt^{\perm(1)}\right) \leq f\left(\Xt + \st \Utt^{\perm(2)}\right) \leq \ldots \leq f\left(\Xt + \st \Utt^{\perm(\lambda)}\right)  .
\end{equation}
To break the possible ties and have an uniquely defined permutation $\perm$, we can simply consider the natural order, i.e.\ if for instance $\lambda=2$ and $f\left( \Xt + \st \Utt^{1} \right) =  f\left( \Xt + \st \Utt^{2} \right)$, then $\perm(1) = 1$ and $\perm(2)=2$. The new estimate of the optimum $\Xtt$ is formed by taking a weighted average of the $\mu$($\geq 1$) best directions (typically $\mu = \lambda/2$), that is
\begin{equation}\label{eq:XxNES}
\Xtt = \Xt + \st \LRm\sum_{i=1}^\mu \weight_i \Utt^{\perm(i)}
\end{equation}
where the sequence of weights $\{ \weight_i: 1 \leq i \leq \mu \}$ sums to $1$ (typically $\weight_1 \geq \ldots \geq \weight_\mu $), and $\LRm > 0$ is called a learning rate.
The step-size is adapted according to 
\begin{equation} \label{eq:sigtxNES}
\stt = \st \exp \left( \frac{\LRsigma}{2n} \left( \sum_{i= 1}^\mu \weight_i \left(\|\Utt^{\perm(i)} \|^2 - n \right)\right)\right) ,
\end{equation}
where $\LRsigma > 0$ is a learning rate for the step-size. The equations \eqref{eq:XxNES} and \eqref{eq:sigtxNES} correspond to the xNES algorithm with covariance matrix restricted to $\st^2 \Idn$ \cite{Glasmachers2010}.


Consider a scaling-invariant function with respect to $\x^*$, that is for all $\rho >0$, $\x, \y \in \R^n$
\begin{equation}\label{eq:SIF}
   f(\x) \leq f(\y)
   \Leftrightarrow
   f( \xstar+ \rho (\x - \xstar)) \leq f( \xstar+ \rho ( \y - \xstar)) .
     \end{equation}
Examples of scaling-invariant functions include $f(\x) = \| \x - \xstar \|$ for any arbitrary norm on $\R^n$. It also includes non continuous functions, functions with non-convex sublevel sets. We assume w.l.g.\ that $x^* = 0$. On this class of functions, $\ZZ:=\{ \Zt = \Xt  / \sigma_k : k \in \Zplus \}$ is a homogeneous Markov chain that can be defined independently of the Markov chain $(\Xt,\sigma_k)$ in the following manner \cite[Proposition 4.1]{methodology-paper}. Given $\Zt \in \R^n$, sample $\lambda$ candidate solutions centered on $\Zt$
using a vector $\Utt$,i.e.\ for $1 \leq i \leq \lambda$ 
\begin{equation}\label{eq:samplingZt}
 \Zt +  \Utt^i ,
\end{equation}
where similarly as for the chain $(\Xt,\st)$, $\{ \Ut : k \in \N\}$ are i.i.d. and each $\Ut$ is a vector of $\lambda$ i.i.d.\ components following each a standard multivariate normal distribution. Those $\lambda$ solutions are evaluated and ranked according to their $f$-values\anne{if $\xstar$ is not zero this would be wrong.}\alex{if $\xstar$ is non-zero, then the ordering must be made on the $f$-values of $\xstar + \Zt + \Utt^i$.}. Similarly to \eqref{eq:selection}, the permutation $\perm$ containing the order of the solutions is extracted. This permutation can be uniquely defined if we break the ties as explained below \eqref{eq:selection}. The update of $\Zt$ then reads
\begin{equation}\label{eq:updateZt}
\Ztt = \frac{\Zt + \LRm \sum_{i=1}^{\mu} \weight_{i} \Utt^{\perm(i)} }{ \exp\left(\frac{\LRsigma}{2 \dim} \left(   \sum_{i=1}^{\mu} \weight_{i}( \| \Utt^{\perm(i)} \|^{2} - \dim) \right)  \right)} .
\end{equation}
Let us now define the vector of selected steps as $\Wtt = (\Utt^{\perm(1)},\ldots,\Utt^{\perm(\mu)}) \in \R^{n  \mu}$ and for $z \in \R^n$, $w \in \R^{n\mu}$ (with $w = (w^1,\ldots,w^\mu) $)
\begin{align}
\G_{\xNES}(z, w)& = \frac{z + \LRm \sum_{i=1}^{\mu} \weight_{i} w^{i}}{\exp\left(\frac{\LRsigma}{2 \dim} \left(   \sum_{i=1}^{\mu} \weight_{i}( \| w^i \|^{2} - \dim) \right)  \right)}  ,
\end{align}
such that
$$
\Ztt = \G_{\xNES}(\Zt, \Wtt) .
$$
The writing of the explicit function $\alpha$ such that $\Wtt=\alpha(\Zt,\Utt)$ is quite tedious in the general case. For the sake of simplicity, we only give it when $\mu=1$ and $\lambda=2$. In this case
\begin{equation}
\Wtt= (U_{k+1}^1 \new{- U_{k+1}^2)}  1_{\{f(\Zt+U_{k+1}^1) \leq f(\Zt+U_{k+1}^2) \}} + U_{k+1}^2
\end{equation}
that is for all $z\in \R^n$ and $u \in \R^{2n} $
$$
\alpha(z,u) = (u^1-u^2) 1_{\{f(z+u^1) \leq f(z+u^2) \}} + u^2 .
$$
\anne{To be discussed with Alexandre: At some point I thought to give the alpha for a (1,lamdbda) but the additional term for breaking the ties seems to be very tedious to write, so I went back to lambda=2 ; Maybe Alexandre you see an easy way to write it, in which case we can go back to the (1,lambda)}\todo{maybe with if condition, it's possible to write}
\alex{
I am not sure why it is not enough to write $\alpha$ as a function of the permutation $S$ that we defined in \eqref{eq:selection}, and then define $\alpha$ as $(u^{S(1)}, \ldots, u^{S(\mu)})$. The permutation is also a function of $z$ and $\u$.

If we want to avoid the permutation, we can define first the rank function:
$$
r_i(z, u^1, \ldots, u^\lambda) := \#\{ j \in [1..\lambda] | f(z + u^j) < f(z + u^i) ~\mathrm{ or }~ f(z + u^j) = f(z + u^i) ~\mathrm{ and}~ j \leq i \},
$$
the second part being to remove ambiguities. Then we define the argrank function through
$$
r_{a_i}(z, u^1, \ldots, u^\lambda) := i
$$
et enfin
$$
\alpha(z, u^1, \ldots, u^\lambda) = (u^{a_1}, \ldots, u^{a_\mu}).
$$
}
The function $\alpha$ is typically discontinuous (similarly to the function $\alpha$ in \eqref{eq:example}). Consider indeed a function $f$ with level sets that are Lebesgue negligible, then if $f(z + u^{1}) = \ldots = f(z + u^{\lambda})$ while the  $\{ u^i: 1 \leq i \leq \lambda\}$ are all distincts, an arbitrarily small change in $u^{1}$ can lead to a different ranking and so to a non continuous change in $\alpha(z, u)$.
%
In the next proposition we derive $p_z(w)$ a density of $\Wtt$ conditional to $\Zt=z$. 
\newcommand{\pn}{p_{\mathcal{N}}}
\begin{proposition} \label{pr:xnesdens}
Let $f:\R^n \to \R$ be an objective function whose level sets are Lebesgue negligible. Let $\lambda \in \Zplusstrict$ and $\mu \in \Zplusstrict$ with $\mu \leq \lambda$. Let us define if $\mu=1$
\begin{equation}\label{dens:one}
p_z(w) = \lambda (1-Q_{z}^{f}(w))^{\lambda-1} \pn(w)
\end{equation}
with $w \in \R^n$, $Q_{z}^{f}(w) = \Pr \left(f(z + \mathcal{N}) \leq f(z+w) \right)  $ with $\mathcal{N} $ following a standard multivariate normal distribution in dimension $n$ and $
\pn(u) = \frac{1}{(\sqrt{2 \pi})^{n}} \exp(- u^T u/2)
$ 
 its density. If $\mu > 1$ 
\begin{equation}\label{dens:mu}
p_z(w) =
\frac{\lambda!}{(\lambda-\mu)!} \1_{\{ f(z+w^1) < \ldots < f(z + w^\mu) \}} (1-Q_{z}^{f}(w^{\mu}))^{\lambda-\mu} \pn(w^{1}) \ldots \pn(w^\mu) ,
\end{equation}
where $w=(w^1,\ldots,w^\mu) \in \R^{n \mu}$.
Then $p_z(w)$ is a density associated to $\alpha(z,U_1)$ (also a density of $\Wtt$ conditionally that $\Zt = z$). 
\end{proposition}
Assume the objective function $f$ is continuous, it is not difficult to see that if $\mu=1$, then $(z,w) \mapsto p_z(w)$ is continuous (and thus lower-semi continuous) and if $\mu > 1$ it is lower semi-continuous.
\anne{if we would only need to consider $\sum wi u^{\perm(i)}$ then we would again be continuous (more than lsc)}

The stability of the homogeneous Markov chain $\ZZ$ is one key to prove the linear convergence of the algorithm defined in \eqref{eq:XxNES} and \eqref{eq:sigtxNES} as stated in the next theorem.
\begin{theorem}[Theorem 5.2 in \cite{methodology-paper}]
Let $f$ be a scaling-invariant function with respect to $0$.  Assume that the Markov chain $\ZZ$ defined in \eqref{eq:samplingZt} and \eqref{eq:updateZt} is $\varphi$-irreducible, Harris-recurrent and positive with invariant probability measure $\pi$. Assume that $E_\pi[|\ln \| z\||] < \infty $ and $E_\pi[\int |\sum_{i=1}^\mu \weight_i (\|w^i\|^2 - n)|p_z(w) dw ] < \infty $, then the \xNES\ algorithm defined in \eqref{eq:XxNES} and \eqref{eq:sigtxNES} converges (or diverges) linearly almost surely, that is for all $X_0$, $\sigma_0$
\begin{equation}\label{eq:cv}
\lim_{k \to \infty}\frac{1}{k} \ln \frac{\| \Xt\|}{\| X_0\|} = \lim_{k \to \infty}\frac{1}{k} \ln \frac{\sigma_k}{ \sigma_0} = E_{z \sim \pi}[ \int \sum_{i=1}^\mu \weight_i (\|w^i\|^2 - n)p_z(w) dw ] .
\end{equation}
\end{theorem}
\renewcommand{\z}{z}
\renewcommand{\u}{u}
Linear convergence happens if the convergence rate $E_{z \sim \pi}[ \int \sum_{i=1}^\mu \weight_i (\|w^i\|^2 - n)p_z(w) dw ] $ is strictly negative. Given that this rate depends on the unknown invariant probability distribution $\pi$, we are often not able to prove the strict negativity. However, it is fairly easy to simulate precisely this convergence rate such that not knowing the sign is generally not problematic.
The proof of this theorem relies on applying a Law of Large Numbers to the chain $\ZZ$.
Hence the stability properties that need to be shown correspond to the assumptions needed for $\ZZ$ to satisfy a Law of Large Numbers. Positivity and Harris recurrence are typically proven by using Foster-Lyapunov drift conditions, that state the negativity of a drift function outside a small set. It is thus important to identify small sets for the chain. \new{Irreducibility and} aperiodicity \new{are}\del{is} also needed because we typically establish a geometric drift and use the geometric ergodic theorem for $\varphi$-irreducible aperiodic chains \cite[Theorem 15.0.1]{meyntweedie}. 
\anne{can I easily show the forward accessibility?}

We will now explain how to use Theorem~\ref{th:ap} to prove that $\ZZ$ is a $\varphi$-irreducible aperiodic $T$-chain and compact sets are small sets for the chain. Remark first that assumption $\mathrm{A1}-\mathrm{A3}$ are satisfied following from the construction and definition of the algorithm, $\mathrm{A4}$ is satisfied as it has been discussed above and the function $F_\xNES$ being $C^1$, the assumption $\mathrm{A5}$ is satisfied. The control sets $\Oxk[1][z]$ for $z \in \R^n$ are defined as $\{ \w \in \R^{n  \mu} | p_z(\w) > 0 \} $ that is for $\mu = 1$, $\Oxk[1][z]=\R^n$ and for $\mu > 1$, $\Oxk[1][z]= \lbrace \w \in \R^{n  \mu} | f(z+w^1) < \ldots < f(z+w^\mu) \rbrace$. We prove in the next proposition that the null vector is a steadily attracting state for CM($F_\xNES$).
\begin{proposition}\label{prop:SA-xNES} 
Let $f:\R^n \to \R$ be continuous with Lebesgue negligible level sets. Then $0$ is steadily attracting for {\rm CM(}$F_\xNES)$.
\end{proposition}
\begin{proof}
We \new{first} assume that $\mu > 1$ \del{. We first} \new{and} prove that for all $\epsilon >0$, for all $y \in \R^n$, there exists a 1-step path from $y$ to $B(0,\epsilon)$ (This latter property implies that $0$ is globally attracting and is actually stronger as the time step to reach the neighborhood of $0$ is independent of the initial point $y$).
%
%
Let $y \in \R^n$, since $\lim_{\|w \| \to +\infty} \G_{\xNES}(\y, w) = 0$, there exists $r > 0$ such that if $\|w\| \geq r$ then $\G_{\xNES}(y, w) \in B(0,\epsilon)$. Let us choose $u \in \R^{n \lambda}$ such that each $u^i$ satisfies $\| u^i \| \geq r$ and additionally $f(y+u^i) \neq f(y+u^j)$ for all $i \neq j$ (we can find such an $u$ because we have assumed that all the level sets of $f$ are Lebesgue negligible). Let $\w_y=\alpha(y,u)$, then $\w_y \in \Oxk[1][y]$ and $\| w_y \| \geq r$ such that $S_y^1(\w_y) \in B(0,\epsilon)$. Hence $\w_y$ is a 1-step path from $y$ to $B(0,\epsilon)$.

Now, showing that a path from $y$ to $B(0,\epsilon)$ exists for all $t \geq 1$ is easy: take $\w \in \Oxk[t-1][y]$, and denote $\tilde{y} = S^{t-1}_y(\w)$; by the previous reasoning, $w_{\tilde{y}}$ is a $1$-step path from $\tilde{y}$ to $B(0, \epsilon)$. Therefore $(\w, \w_{\tilde{y}})$ is a $t$-steps path from $y$ to $B(0, \epsilon)$, and so $0$ is steadily attracting.  

In the case where $\mu=1$, the proof is even simpler as we do not need to care for finding a step that derives from a vector $o \in \R^{n \lambda}$ that does not result in solutions on the same $f$-level sets. We omit the details than can be easily deduced from the previous case.
\end{proof}
In the previous proof we have shown that any neighborhood of $0$ can be reached via a $1$-step path from any starting point. This directly implies that $0$ is steadily attracting. More generally, if we can reach any neighborhood of $\xstar$ in $T$ steps from any initial point $x$---that is $T$ is independent of the initial point $x$---then $\xstar$ is steadily attracting. 

An other practical result to prove that a state $\xstar$ is steadily attracting \new{holding under a controllability condition} is first to prove that the point is globally attracting and then to show that there exists $\w \in \Oxk[1][\xstar]$ that allows to stay in $\xstar$, that is such that $\Sxk[1][\xstar](\w) = \xstar$. Those two practical conditions to prove that a state $\xstar$ is steadily attractive are formalized in the following lemma.
\begin{lemma}[Practical Conditions for a Steadily Attracting State] \label{lem:PCSA}
\new{Suppose that conditions $\mathrm{A1}-\mathrm{A4}$ hold, and that $F$ is continuous.}
Let $\xstar \in \xdom$, the following holds:\\
$\mathrm{(i)}$ If for all $\Uo_{\xstar}$ neighborhood of $\xstar$, there exists $T \in \Zplusstrict$ such that for any $\y \in \xdom$ 
there exists a $T$-steps path from $y$ to $\Uo_{\xstar}$, then $\xstar$ is steadily attracting.\\
$\mathrm{(ii)}$ \new{Suppose that assumption $\mathrm{A5}$ holds and that for all $x \in \xdom$ there exists $k \in \Zplusstrict$ and $\w \in \Oxk$ for which $\rank \Cxk(\w) = n$.} If $\xstar$ is globally attracting and if there exists $\w^* \in \Oxk[1]$  that allows to stay in $\xstar$, that is such that  $\Sxk[1][\xstar](\w^*) = \xstar$, then $\xstar$ is steadily attracting.
\end{lemma}
The proof of the second point is not completely straightforward and is presented in the appendix \new{as a consequence of Proposition~\ref{pr:saab}}.

We have now seen that $0$ is a steadily attracting state. We will prove that the rank condition is satisfied in $0$. We prove more precisely the following proposition.
\begin{proposition}
Let $f:\R^n \to \R$ be continuous with Lebesgue negligible level sets, then there exists $\w^*$ in $ \Oxk[1][0]$ such that rank $\Cxk[1][0](\w^*)$ equals $n$, that is the rank condition is satisfied at the steadily attracting state $0$.
\end{proposition}
\begin{proof}
We  prove that there exists $\w^* \in \Oxk[1][0]$ such that $\Cxk[1][0](\w^*)$ has rank $n$. Let $\w_0 = (0, \ldots, 0) \in \R^{n \mu}$, we will first prove that $\Cxk[1][0](\w_0)$ has rank $n$. We will for this prove that the differential of $w \to \Sxk[1][0](w) := \G_{\xNES}(0, w)$ at $\w_0$ is surjective. Let $h = (h_i)_{i = 1, \ldots, \mu} \in \R^{n \mu}$, then
\begin{align*}
\Sxk[1][0] (\w_0 + h) &= \frac{\LRm \sum_{i=1}^{\mu} \weight_{i} h_i}{\exp\left(\frac{\LRsigma}{2 \dim} \left(   \sum_{i=1}^{\mu} \weight_{i}( \| h_{i} \|^{2} - \dim) \right)  \right)}  \\
  &= \Sxk[1][0] (\w_0) + \LRm \exp\left(\frac{\LRsigma}{2}\right) \left( \sum_{i=1}^{\mu} \weight_{i} h_i \right)  \left( 1 + o(\|h\|) \right)  .
\end{align*}
\mathnote{$=   0 + \LRm \exp\left(\frac{\LRsigma}{2}\right) \left( \sum_{i=1}^{\mu} \weight_{i} h_i \right) \exp\left(-\frac{\LRsigma}{2 n } \sum_{i=1}^{\mu} \weight_{i} \| h_{i} \|^{2} \right)$}
Hence $D \Sxk[1][0](\w_0) (h) = \LRm \exp\left(\frac{\LRsigma}{2}\right) \sum_{i=1}^{\mu} \weight_i h_i$ which is a surjective linear map, which implies that the rank of $\Cxk[1][0](\w_0)$ is $n$. The point $\w_0$ is not in $\Oxk[1][0]$, but since $w \to \Sxk[1][0](w)$ is $C^1$, there exists $\Vopen_{\w_0}$ an open neighborhood of $\w_0$ such that for all $\v \in \Vopen_{\w_0}$, $\Cxk[1][0](\v)$ has rank $n$. Finally since $\Vopen_{\w_0} \cap \Oxk[1][0]$ is not empty (since $f$ has Lebesgue negligible level sets, we can find $\mu$ distinct points arbitrarily close to zero with different $f$-values, the ranked vectors will belong to $\Oxk[1][0]$), there exists $\w^* \in \Oxk[1][0]$ such that $\Cxk[1][0](\w^*)$ has rank $n$.
\end{proof}
The two previous lemmas prove that $0$ is  a steadily attracting state where the rank condition on the controllability matrix is satisfied. Hence according to Theorem~\ref{th:ap}, the Markov chain $\ZZ$ defined in \eqref{eq:updateZt} is an aperiodic $\varphi$-irreducible T-chain and every compact set is small.

\begin{keyword}
\kwd{Markov Chains}
\kwd{Irreducibility}
\kwd{Aperiodicity}
\kwd{T-chain}
\kwd{Control model}
\kwd{Optimization}
\end{keyword}


\renewcommand{\z}{\bs{z}}
\renewcommand{\u}{\bs{u}}

\subsection*{Acknowledgements}
We are very grateful to the anonymous reviewers for their constructive comments that helped us to improve considerably the presentation of the results. A special thanks goes to Reviewer 3 who pointed out a much simpler approach allowing to present nicer proofs. His comments made us realize several important connections with previous results that we had  missed in the first version of the paper and that are now presented.

\section{Appendix} \label{sec:appendix}

\subsubsection*{Proof of Proposition~\ref{pr:GA}}

It is immediate that if $\xstar \in \xdom$ is globally attractive then $\mathrm{(i)}$ holds. Indeed for all $y \in \xdom$, \eqref{eq:GA} implies that $\xstar \in \overline{ \bigcup_{k = 1}^{+\infty} A_+^k(\y) }  \subset \overline{A_+(\y)}$.

\newcommand{\ko}{{k_0}}

Now we show that $\mathrm{(i)}$ implies $\mathrm{(ii)}$. Suppose that $\mathrm{(i)}$ holds, take $\y \in \xdom$, $\Uopen$ an open set containing $\xstar$, and $u \in \Oxk[1][y]$. Since from $\mathrm{(i)}$, $\xstar \in \overline{A_+(\Sxk[1][y](u))}$, there exists $z \in A_+(\Sxk[1][y](u))$ such that $z \in \Uopen$. Either $z \in A_+^0(\Sxk[1][y](u))=\{ \Sxk[1][y](u) \}$, then $u$ is a $1$-step path from $y$ to $\Uo$ or $z \in A_+^k(\Sxk[1][y](u))$ for $k > 0 $ but then there exists $\w \in \Oxk[k][{\Sxk[1][y](u)}]$ such that $z = S_{\Sxk[1][y](u)}^k(\w) = \Sxk[k+1][y](u, \w)$ and thus $(u,\w)$ is a $k+1$ path from $y$ to $\Uo$.

Now we show that $\mathrm{(ii)}$ implies $\mathrm{(iii)}$. Suppose that $\mathrm{(ii)}$ holds. Hence there exists $k_1 \in \Zplusstrict$ and $\w_{1}$ a $k_1$-steps path from $\y$ to $B(\xstar, 1)$. Let $\y_{1}$ denote $\Sxk[k_1][y](\w_{1})$. Inductively for $t \in \Zplusstrict$ there exists $k_{t+1}$ and $\w_{t+1}$ a $k_{t+1}$-steps path from $\y_{t}$ to $B(\xstar, 1/(t+1))$, and we define $\y_{t+1}$ as $\Sxk[k_{t+1}][\y_t](\w_{k+1})$. We then have $\y_{t} \in A_+^{k_1 +\ldots+ k_t}(\y)$ with $k_i >0$ for $i = 1, \ldots, t$, so $\{\y_t: t \in \Zplusstrict\}$ is a subsequence of a sequence of $\prod_{i \in \Zplusstrict} A_+^i(\y)$. Finally $\y_t \in B(\xstar, 1/t)$ so this subsequence converges to $\xstar$.

Finally we show that $\mathrm{(iii)}$ implies that $\xstar$ is a globally attracting state. Suppose that $\mathrm{(iii)}$ holds, that is for all $y \in \xdom$  there exists $\{\y_k: k \in \Zplusstrict\}$ a sequence \new{with $y_k \in A_+^k(\y)$}\del{of $\prod_{k \in \Zplusstrict} A_+^k(\y)$}  from which we can extract a subsequence converging to $\xstar$. Since $\y_k \in A_+^k(\y) \subset \bigcup_{i = k}^{\infty} A_+^i(\y)$, and that for any $k \in \Zplusstrict$, the state $\xstar$ is the limit of a subsequence of $\{\y_i: i \geq k\}$, we have $\xstar \in \overline{\bigcup_{i \geq k} A_+^i(\y)}$ for any $k \in \Zplusstrict$, and so \eqref{eq:GA} holds for all $y \in \xdom$.
\qed

\subsubsection*{Proof of Proposition~\ref{pr:pointproba}}
Let $x \in \xdom$, $\Oo \in \borel(\xdom)$ be an open set, $k \in \Zplusstrict$ and $\w \in \Oxk$ be a $k$-steps path from $x$ to $\Oo$. \del{Since the function $(x,w) \mapsto p_x(w)$ is lower semi-continuous and $F$ is continuous, according to lemmas~\ref{lm:scp} and \ref{lm:plsc}, the functions $(x, \u) \mapsto \pxk(\u)$ and $(x, \u) \mapsto\Sxk(\u)$ are also respectively lower semi-continuous and continuous.} From the continuity of $\Sxk$ \new{(by Lemma~\ref{lm:scp})} there exists $\eta_1 > 0$ such that for all $\u \in B(\w, \eta_1)$, $\Sxk(\u) \in \Oo$. Since $\w \in \Oxk$, $p_0 := \pxk(\w) > 0 $ and from the lower semi-continuity of $\pxk$ \new{(by Lemma~\ref{lm:plsc})}, there exists $\eta_2 > 0$ such that for all $\u \in B(\w, \eta_2)$, $\pxk(\u) > p_0 / 2$. Hence
\begin{align*}
P^k(x, \Oo) &= \int_{\Oxk} \1_{\Oo}(\Sxk(\u)) \pxk(\u) d\u \\
 & \geq \int_{B(\w, \min(\eta_1, \eta_2))} \frac{p_0}{2} d\u > 0 .
\end{align*}
Conversely, for any $x \in \xdom$, $k \in \Zplusstrict$ and $A \in \borel(\xdom)$
$$
P^k(x, A) = \int_{\Oxk} \1_{A}(\Sxk(\u)) \pxk(\u) d\u
$$
and therefore if $P^k(x,A)>0$, then there exists $\u$ such that $\pxk(\u)>0$ and $\Sxk(\u) \in A$ that is, there exists a $k$-steps path from $x$ to $A$.
\qed

\subsubsection*{Proof of Corollary~\ref{cr:GAreachability}}

Let $\xstar$ be a reachable point. Then for all $y \in \xdom$ and $\Oo \in \mathscr{B}(\xdom)$ open containing $\xstar$, there exists $k \in \Zplusstrict$ such that $P^k(y, \OOpen) > 0$. Hence from Proposition~\ref{pr:pointproba} there exists a $k$-steps from $y$ to $\Oo$, and from Proposition~\ref{pr:GA}, $\xstar$ is globally attracting.

Conversely, let $\xstar$ be globally attracting. From Proposition~\ref{pr:GA}, for all $y \in \xdom$ and $\Oo$ open neighborhood of $\xstar$, there exists $k \in \Zplusstrict$ and a $k$-steps path from $y$ to $\Oo$. From Proposition~\ref{pr:pointproba} we have $P^k(y, \Oo)> 0$, and so $\xstar$ is reachable.
\qed

\subsubsection*{Proof of Proposition~\ref{pr:SGA}}

Let $\xstar$ be a steadily attracting state, then from Proposition~\ref{pr:GA} ($\mathrm{ii}$), we immediately find that $\xstar$ is globally attracting.

We now prove $\mathrm{(ii)}$. Suppose that $\xstar$ is steadily attracting, and let $\y \in \xdom$. We will construct a sequence \new{$(y_k)_{k \in \Zplusstrict}$ with $y_k \in A_+^k(y)$} \del{of $\prod_{k \in \Zplusstrict} A_+^k(y)$} converging to $\xstar$. There exists $(T_i)_{i \in \Zplusstrict}$ a strictly increasing sequence of integers such that for all $t \geq T_i$, there exists a $t$-steps path from $\y$ to $B(\xstar, 1/i)$. For any $k \in \Zplusstrict$, we construct the sequence $\y_k$ in the following way. If $k < T_1$, let $\y_k$ be any point of $A_+^k(y)$. Else for the largest $i \in \Zplusstrict$ such that $T_i \leq k$, there exists $\w_k$ a $k$-steps path from $y$ to $B(\xstar, 1/i)$. Let $y_k = \Sxk[k][y](\w_k)$. Then $\{ y_k : k \in \Zplusstrict\}$ \new{with $y_k \in A_+^k(y)$} is a sequence \del{of $\prod_{k \in \Zplusstrict} A_+^k(y)$} converging to $\xstar$.

Now suppose that for all $\y \in \xdom$ there exists a sequence $\{ \y_k : k \in \Zplusstrict\}$ with $y_k \in A_+^k(y)$ which converges to $\xstar$, and take $\Uo$ a neighborhood of $\xstar$. There exists $T$ such that for all $k \geq T$, $y_k \in \Uo$. Since $y_k \in A_+^k(y)$, there exists $\w_k \in \Oxk[k][y]$ such that $\Sxk[k][y](\w_k) = y_k \in \Uo$. Hence $\w_k$ is a $k$-steps path from $y$ to $\Uo$. Since such a $\w_k$ exists for all $k \geq T$, this proves that $\xstar$ is steadily attracting.

We now prove $\mathrm{(iii)}$. Suppose that $\xstar$ is steadily attracting and that $\ystar$ is globally attracting. Then for $\Uo_{\ystar}$ an open neighborhood of $\ystar$, there exists $k \in \Zplusstrict$ and $\w \in \Oxk[k][\xstar]$ such that $\Sxk[k][\xstar](\w) \in \Uo_{\ystar}$. \new{By Lemma~\ref{lm:scp} and \ref{lm:plsc}} \del{Since} the function $x\mapsto \pxk$ is lower semi-continuous and the function $\x \mapsto \Sxk(\w)$ is continuous\new{, so since} \del{and} $\Uo_{\ystar}$ \new{is} open there exists $\epsilon >0$ such that for all $x \in B(\xstar, \epsilon)$, $\w \in \Oxk$ and $\Sxk(\w) \in \Uo_{\ystar}$. And as $\xstar$ is steadily attracting, for all $z \in \xdom$ there exists $T \in \Zplusstrict$ such that for all $t \geq T$ there exists $\u_t \in \Oxk[t][z]$ for which $\Sxk[t][z](\u_t) \in B(\xstar, \epsilon)$. Therefore $\Sxk[t+k][z](\u_t, \w) \in \Uo_{\ystar}$ for all $t \geq T$, that is $\ystar$ is steadily attracting.

\qed

\subsubsection*{Proof of Proposition~\ref{pr:EA2AA}}

We first prove $\mathrm{(i)}$. Since $\xstar$ is attainable, according to \eqref{eq:blabla}, there exists $a \in \Zplusstrict$ and $\w_a \in \Oxk[a][\xstar]$ such that $\xstar = \Sxk[a][\xstar](\w_a)$. Hence for all $t \geq 0$, $\xstar = \Sxk[at][\xstar](\w_a, \ldots, \w_a) \in A_+^{at}(\xstar)$, meaning $a \in E$ and so $E$ is not empty.

Take $(a, b) \in E^2$, and denote $d$ the greatest common divisor of $a$ and $b$. There exists $(t_a, t_b) \in \Zplus$ such that for all $t \geq t_a$, $\xstar \in A_+^{at}(\xstar)$ and for $t \geq t_b$, $\xstar \in A_+^{bt}(\xstar)$. To show that $d \in E$, we will show that there exists $k_0 \in \Zplus$ such that for all $t \geq 0$, there exists $(c_1, c_2) \in \Zplus^2$ for which 
\begin{equation}\label{eq:int1}
(k_0 + t)d = a c_1 + b c_2 .
\end{equation} 
Indeed if this holds as $\xstar \in A_+^{a (c_1 + t_a)}(\xstar)$ and $\xstar \in A_+^{b (c_2 + t_b)}(\xstar)$, $\xstar \in A_+^{(k_0 + t)d + at_a + bt_b}(\xstar)$. Since $d$ divides $a$ and $b$, $\xstar \in A_+^{(k_0+t+p)d}(\xstar)$ for some $p \in \Zplus$. Hence for all $t \geq k_0 + p$, $\xstar \in A_+^{td}(\xstar)$, i.e.\ $d \in E$.

It remains to show \eqref{eq:int1}. According to B\'{e}zout's identity there exists $(u, v) \in \mathbb{Z}^2$ for which $au + bv = d$. If $u$ or $v$ is zero, then $d$ equals $b$ or $a$ and so $d \in E$. Else, w.l.o.g. suppose that $v<0$ (and so $u > 0$). Take $k \geq - v a / d$ and $t \in \Zplus$. We will write $kb+ td$ as a positive sum of $a$ and $b$. As a result of Euclidian division of $t$ by $a/d$ there exists $q \in \Zplus$ and $s \in \lbrace  0, \ldots, a/d-1\rbrace$ such that $t d = a q + s d$, and so $kb + td = kb + qa + sd$. Furthermore $sd = s(au + bv)$ so $kb + td = (k+sv)b + (q + su)a$, with the coefficients $k+sv$ and $q+su$ positive. Hence for all $t \geq 0$, $(kb/d + t)d = (k+sv)b + (q+su)a$ and hence we have proven \eqref{eq:int1} with $c_1=q+su$, $c_2=k+sv$.
 
 

We now prove $\mathrm{(ii)}$. If $\gcd(E) = 1$, then $1 \in E$, i.e.\ there exists $t_0 \in \Zplusstrict$ such that for all $t \geq t_0$, $\xstar \in A_+^{t}(\xstar)$. Since $\xstar$ is attainable, for all $y \in X$ there exists $t_1 \in \Zplusstrict$ such that $\xstar \in A_+^{t_1}(y)$, and so for all $t \geq t_0 + t_1$, $\xstar \in A_+^t(y)$. Therefore for all $\Uo$ neighborhood of $\xstar$ and for all $t \geq t_0 + t_1$ there exists $\w \in \Oxk[t][y]$ such that $\Sxk[t][y](\w) = \xstar \in \Uo$, meaning $\xstar$ is steadily attracting.

We finally prove $\mathrm{(iii)}$. Let $d := \gcd(E)$ and for $i = 0, \ldots , d-1$ take $D_i := \bigcup_{r \in \Zplus} A_+^{rd + i}(\xstar)$. Then
\begin{itemize}
\item for all $y \in A_+^t(\xstar)$, $P(y, A_+^{t+1}(\xstar)) = 1$, so for $i = 0, \ldots, d-1$ and $y \in D_i$, $P(y, D_{i+1 \mod d}) = 1$.
\item take $y \in A_+^i(\xstar)$. Since $\xstar$ is attainable there exists $t_0 \in \Zplusstrict$ such that $\xstar \in A_+^{t_0}(y)$, so $\xstar \in A_+^{t_0 + i}(\xstar)$ and $\xstar \in A_+^{r(t_0 + i)}(\xstar)$ for all $r \geq 0$, i.e.\ $t_0 + i \in E$. As $d = \gcd(E)$, this implies that $i = k_1d - t_0$ for some $k_1 \in \Zplusstrict$. By the same reasoning if $y \in A_+^j(\xstar)$ for some $j \in \Zplus$, then for some $k_2 \in \Zplusstrict$, $j = k_2 d - t_0$ also. Therefore $i - j$ is a multiple of $d$, which implies that the sets $(D_i)_{i = 0, \ldots, d-1}$ are disjoint sets.
\item $\bigcup_{i=0}^{d-1} D_i = A_+(\xstar)$. Since for all $k \in \Zplusstrict$, $P^k(\xstar, A_+(\xstar)^c) = 0$, $\varphi((\bigcup_{i=0}^{d-1} D_i)^c) = 0$ for all $\varphi$ irreducibility measure.
\end{itemize}
Hence $(D_i)_{i = 0, \ldots, d-1}$ is a $d$-cycle.
\qed

\subsubsection*{Derivation of \eqref{eq:Cpartial}:}
By chain rule, for $i = 1, \ldots, k-1$,
\begin{align*}
\left[ \frac{\partial \Sxk[k][\new{y}\del{x}]}{\partial w_i} \right]_{\del{(x, }(w_1, \ldots, w_k)\del{)}}  &= \left[\frac{\partial F}{\partial x} \right]_{(\Sxk[k-1][\new{y}\del{x}], w_k)} \left[\frac{\partial \Sxk[k-1][\new{y}\del{x}]}{\partial w_{i}}\right]_{\del{(x,} (w_1, \ldots, w_{k-1})\del{)}} \\
 &= A_{k-1} \left[\frac{\partial \Sxk[k-1][\new{y}\del{x}]}{\partial w_{i}}\right]_{\del{(x,} (w_1, \ldots, w_{k-1})\del{)}}  \\
 &= A_{k-1} \ldots A_i \left[ \frac{\partial \Sxk[i][\new{y}\del{x}]}{\partial w_i} \right]_{\del{(x,} (w_1, \ldots, w_i)\del{)}} .
\end{align*}
Furthermore, for $i = 1, \ldots, k$,
$$
\left[ \frac{\partial \Sxk[i][\new{y}\del{x}]}{\partial w_i} \right]_{\del{(x,} (w_1, \ldots, w_i)\del{)}} = \left[\frac{\partial F}{\partial w} \right]_{(\Sxk[i-1][\new{y}\del{x}], w_i)} = B_{i-1} ,
$$
and so
$$
\left[ \frac{\partial \Sxk[k][\new{y}\del{x}]}{\partial w_i} \right]_{\del{(x,} (w_1, \ldots, w_k)\del{)}} = A_{k-1} \ldots A_{i} B_{i-1}.
$$

\subsubsection*{Two lemmas}

We state two simple lemmas whose results are often used. The first one is given without any proof.
\begin{lemma} \label{lm:scp}
Suppose that the function $(x,w) \mapsto F(x, w) $ is $C^m$ for $m \in \N$, then for all $k \in \Zplusstrict$, the function $(x, \w) \mapsto \Sxk(\w)$ defined in \eqref{eq:sxk} is $C^m$.
\end{lemma}
\del{\begin{proof}
By hypothesis, the function $(x, w) \mapsto \Sxk[1](w) = F(x,w)$ is $C^m$. Suppose that $(x,w_1, \ldots, w_k) \mapsto \Sxk(w_1, \ldots, w_k)$ is $C^m$. \alex{could be replaced with $w_1^k$ notation} Then the function $h : (x, w_1, \ldots, w_{k+1}) \mapsto (\Sxk(w_1, \ldots, w_k), w_{k+1})$ is $C^m$, and so is $(x, w_1, \ldots, w_{k+1}) \mapsto \Sxk[k+1](w_1, \ldots, w_k) = F \circ h(x, w_1, \ldots, w_{k+1})$.
\end{proof}}

\begin{lemma} \label{lm:plsc}
Suppose that the function $p: (x, w) \mapsto p_x(w)$  is lower semi-continuous and the function $(x,w) \mapsto F(x, w) $ is continuous, then for all $k \in \Zplusstrict$ the function $(x, \w) \mapsto \pxk(\w)$ defined in \eqref{eq:pxk} is lower semi-continuous.
\end{lemma}

\begin{proof}
According to Lemma~\ref{lm:scp}, the function $(x,\w) \mapsto \Sxk(\w)$ is continuous, and by hypothesis, the function $p$ is lower semi-continuous. Now suppose that the function $(x, \w) \mapsto \pxk(\w)$ is lower semi-continuous. The function $(x, \w, u) \mapsto p_{{\Sxk(\w)}}(u)$ is lower semi-continuous as the composition of a continuous and a lower semi-continuous function. So by ~\eqref{eq:pxk} the function $(x, \w, u) \mapsto \pxk[k+1](\w, u)$ is lower semi-continuous as the product of two non-negative lower semi-continuous functions (see for instance \cite[Proposition~B.1]{puterman}).
\mathnote{Take $f$ and $g$ two non-negative functions lower semi-continuous at $x$, and take $\epsilon > 0$. If $\epsilon \geq f(x)g(x)$ then for all $y$, $f(y)g(y) \geq  0 \geq f(x)g(x) - \epsilon$. Else, $f(x)$ and $g(x)$ are strictly positive and there exists $\alpha > 1$ such that $f(x)g(x) / \alpha \geq f(x)g(x) - \epsilon$. By lower semi-continuity of $f$ and $g$ at $x$, there exists $\eta>0$ such that for all $y \in B(x, \eta)$, $f(y) > f(x)/\sqrt{\alpha}$ and $g(y) > g(x)/\sqrt{\alpha}$. Therefore $f(y)g(y) \geq f(x)g(x)/\alpha \geq f(x)g(x) - \epsilon$, which shows the lower semi-continuity of $fg$ at $x$.}
\end{proof}

\subsubsection*{Proof of Proposition~\ref{pr:gacontrol}}

The matrix $\Cxk[k][\xstar](\wstar)$ has rank $n$ and thus with the expression of the matrix given in \eqref{eq:Cpartial}, we can find $i_1, \ldots, i_n$ such that
$$
\det \left[ \frac{\partial S_{x^*}^k}{\partial w_{i_1}} | \ldots | \frac{\partial S_{x^*}^k}{\partial w_{i_n}} \right]_{\wstar} \neq 0 . 
$$
\newcommand{\Nopen}{\mathscr{N}}
Since the function $(x, \w) \mapsto \Sxk(\w)$ is a $C^1$ function (as a consequence of the continuity of $F$ and Lemma~\ref{lm:scp}), there exists $\Nopen_{(\xstar, \wstar)}$ a neighborhood of $(\xstar, \wstar)$ such that for all $(x, \w) \in \Nopen_{(\xstar, \wstar)}$
$$
\det \left[ \frac{\partial S_{x}^k}{\partial w_{i_1}} | \ldots | \frac{\partial S_{x}^k}{\partial w_{i_n}} \right]_{\w} \neq 0 . 
$$

Let $r_1 > 0 $ and $r_2 > 0$ such that $B(\xstar, r_1) \times B(\wstar, r_2) \subset \Nopen_{(\xstar, \wstar)}$.
 Since $\xstar$ is a globally attracting state, for all $y \in \xdom$ according to Proposition~\ref{pr:GA} there exists $t_0 \in \Zplusstrict$ and $\uu_0$ a $t_0$-steps path from $y$ to $B(\xstar, r_1)$. Since $(\Sxk[t_0][y](\u_0), \wstar) \in \Nopen_{(\xstar, \wstar)}$, the matrix $\Cxk[k][{\Sxk[t_0][y](\u_0)}](\w^*)$ has rank $n$. Hence, taking $T = t_0 + k$ and $\uu = (\uu_0, \wstar)$, according to \eqref{eq:controlrec} the matrix $\Cxk[T][y](\uu)$ also has rank $n$.
\qed

\subsubsection*{Proof of Proposition~\ref{pr:GAEA}}

This result is a consequence of the implicit function theorem. Take $\xstar \in \xdom$, $k\in \Zplusstrict$ and $\wstar \in \Oxk[k][\xstar]$ for which the controllability matrix $\Cxk[k][\xstar](\wstar)$ has rank $n$. We first prove that there exists $\Uo_{\xstar}$ a neighborhood of $\xstar$ such that for all $y \in \Uo_{\xstar}$, there exists $\u \in \Oxk[k][y]$ such that $\Sxk[k][y](\u) = \Sxk[k][\xstar](\wstar)$.
Since for all $m \in \Zplusstrict$ the function $(x,\w) \mapsto \pxk[m](\w)$ is lower semi-continuous (as a consequence of Lemma~\ref{lm:plsc}), there exists $\Oo$ an open neighborhood of $(\xstar, \wstar)$ such that for all $(x, \w) \in \Oo$, $\pxk(\w) > 0$. Since $\rank(\Cxk[k][\xstar](\wstar)) = n$, using the expression of the controllability matrix given in \eqref{eq:Cpartial} (where $w_i$ are coordinates rather than vectors), there exists integers $({i_1}, \ldots, {i_n})$ such that
$$
\det \left[ \frac{\partial \Sxk[k][\xstar]}{\partial w_{i_1}} | \ldots | \frac{\partial \Sxk[k][\xstar]}{\partial w_{i_n}} \right]_{\wstar} \neq 0 .
$$

Assume that $i_1 = kp - n + 1, \ldots, i_n = kp$ (which can be imposed by considering a composition of $\Sxk$ with a function permuting the variables). The partial differential of the $C^1$ function $(x, w_1, \ldots, w_{kp}) \mapsto \Sxk(w_1, \ldots, w_{kp})$ \todo{I think you kind of change notation here, before $w_1$ was a vector of $\R^p$ (first coordinate of $\w$) now it's just a coordinate, right?} with respect to $(w_{kp -n +  1}, \ldots, w_{kp})$ is invertible at $(\xstar, \wstar)$, so according to the implicit function theorem applied to this function restricted to $\Oo$, there exists $\Uo$ and $\Vo$ open neighborhoods of respectively $(\xstar, w_1^*, \ldots, w_{kp -n}^*)$ and $(w_{kp -n +1}^*, \ldots, w_{kp}^*)$ such that $\Uo \times \Vo \subset \Oo$, and a $C^1$ function $g : \Uo \to \Vo$ such that $\Sxk(w_1, \ldots, w_{kp-n}, g(x, w_1, \ldots, w_{kp-n})) = \Sxk[k][\xstar](\wstar)$ \anne{So we apply the implicit function theorem to the function $(w_1,\ldots,w_{kp}) \to \Sxk(w_1, \ldots, w_{kp}) -  \Sxk[k][\xstar](\wstar) $, right?}\alex{yes, the implicit function theorem is specified on wiki for $\0$, but with the trick you gave it's easy to see that it holds for any point reached.} for all $(x, w_1, \ldots, w_{kp-n}) \in \Uo$, which proves our first point by taking $\Uo_{\xstar} = \lbrace x \in \xdom | (x, w_1, \ldots, w_{kp -n}) \in \Uo\rbrace$.

Now suppose that $\xstar$ is globally attracting. Then according to Proposition~\ref{pr:GA} $\mathrm{(ii)}$,  for all $z \in \xdom$ there exists $t_0 \in \Zplusstrict$ and $\v \in \Oxk[t_0][z]$ such that $\Sxk[t_0][z](\v) \in \Uo_{\xstar}$. And so there exists $\u \in \Oxk[k][{\Sxk[t_0][z](\v)}]$ such that $\Sxk[t_0 + k][z](\v, \u) = \Sxk[k][\xstar](\wstar)$.
\qed

\subsubsection*{Proof of Proposition~\ref{pr:forwardaccess}}

Let $\xstar \in \xdom$ such that there exists $k \in \Zplusstrict$ and $\wstar = (w^*_1, \ldots, w^*_{kp}) \in \Oxk[k][\xstar]$ such that $\rank(\Cxk[k][\xstar](\wstar)) = n$, and let us prove that $A_+^k(\xstar)$ contains an open set (which would imply the forward accessibility of the control model, given that the condition holds for all $\xstar \in \xdom$). 

Since $\rank(\Cxk[k][\xstar](\wstar)) = n$, following \eqref{eq:Cpartial} there exists integers $\lbrace i_1, \ldots, i_n \rbrace$ such that
$$
\det \left[ \frac{\partial \Sxk[k][\xstar]}{\partial w_{i_1}} | \ldots | \frac{\partial \Sxk[k][\xstar]}{\partial w_{i_{n}}} \right]_{\wstar} \neq 0 .
$$
Assume that $i_1 = kp - n + 1$, $\ldots$, $i_n = kp$ (which can be imposed by considering the composition of $\Sxk[k][\xstar]$ with a function permuting the variables), let $G$ denote the $C^1$ function $(u_1, \ldots, u_n) \in \R^n \mapsto \Sxk[k][\xstar](w^*_1, \ldots, w^*_{kp-n}, u_1, \ldots, u_n)$ and $\Oo$ the set $\lbrace (u_1, \ldots, u_n) \in \R^n | (w^*_1, \ldots, w^*_{kp-n}, u_1, \ldots, u_n) \in \Oxk[k][\xstar] \rbrace$ (which is open since $\Oxk[k][\xstar]$ is open). 
Then the Jacobian determinant of $G$ is non-zero, and according to the inverse function theorem applied to $G$ restricted to $\Oo$ there exists $\Uo \subset \Oo$ 
and $\Vo$ open neighborhoods of respectively $(w^*_{kp-n+1}, \ldots, w^*_{kp})$ and $\Sxk[k][\xstar](\wstar)$, such that $G$ is a bijection from $\Uo$ to $\Vo$. \del{Hence for all $y \in \Vo$ there exists $(u_1, \ldots, u_n) \in \Uo$ such that $\Sxk[k][\xstar](w^*_1, \ldots, w^*_{kp-n}, u_1, \ldots, u_n) = y$ and since $\Uo \subset \Oo$, $(w^*_1, \ldots, w^*_{kp-n}, u_1, \ldots, u_n) \in \Oxk[k][\xstar]$ and so  $y \in A_+^k(\xstar)$.}\alex{seems obvious enough} Therefore $\Vo \subset A_+^k(\xstar)$ which hence has non empty interior.


Suppose now that $F$ is $C^\infty$ and that the control model is forward accessible. Then, for all $\x \in \xdom$, $A_+(\x)$ is not Lebesgue negligible. Since $\sum_{k \in \Zplus} \mleb (A_+^k(\x)) \geq \mleb(A_+(\x)) > 0$, there exists $k \in \Zplusstrict$ such that $\mleb(A_+^k(\x)) > 0$ ($k \neq 0$ because $A_+^0(\x) = \lbrace \x \rbrace$ is Lebesgue negligible). Let $N_x^k$ denote the set $\lbrace \w \in \Oxk | \w \textrm{ is a critical point of } \Sxk \rbrace$ (i.e.\  $\w \in N_x^k$ if and only if $\rank(\Cxk(\w)) < n$). According to Lemma~\ref{lm:scp} the function $\Sxk$ is $C^\infty$, so we can apply Sard's theorem \cite[Theorem II.3.1]{zbMATH03210572} to $\Sxk$ which implies that the image of its critical points is Lebesgue  negligible, i.e.\ $\mleb(\Sxk(N_x^k)) = 0$. Since $\mleb(\Sxk(N_x^k)) < \mleb( A_+^k(\x))$ and $A_+^k(x) = \Sxk(\Oxk)$, $N_x^k$ is a strict subset of $\Oxk$ meaning there exists $\w \in\Oxk$ for which $\rank(\Cxk(\w)) = n$.
\mathnote{Generalization: if $F$ is $C^{kp - n + 1}$ and if $A+^k(x)$ has non-empty interior for some $x$ then there exists $\w \in \Oxk$ such that $\Cxk(w)$ has rank $n$.}
\qed

\subsubsection*{Proof of Proposition~\ref{pr:2.1}}

In order to prove the proposition, we use the following lemma, which is identical to \cite[Lemma~3.0]{meyn-caines1991} with the exception that the function $G$ is here assumed to be $C^1$ instead of $C^\infty$\new{, and that $\R^n$ has in some places been replaced with $\Xo_1$}. These changes do not impact the proof of the lemma. Variable names have also been changed for consistency with the notations used in this article.

\begin{lemma}[{\cite[Lemma~3.0]{meyn-caines1991}}] \label{lm:3.0}
Let $\Xo_1 \subset \R^n$, $\widetilde{\Wo}_1 \subset \R^m$ and $\widehat{\Wo}_1 \subset \R^n$ be open sets and let $G: (x,\tilde{w},\hat{w}) \in \Xo_1 \times \widetilde{\Wo}_1 \times \widehat{\Wo}_1 \mapsto z \in \del{\R^n}\new{\Xo_1}$ be a $C^1$ function such that the matrix $\partial G / \partial \hat{w}$ has rank $n$ at some $(x_0, \tilde{w}_0, \hat{w}_0) \in \Xo_1 \times \widetilde{W}_1 \times \widehat{W}_1$. Then
\begin{enumerate}[(i)]
\item	there exists an open set $\Xo \times \widetilde{W} \times \widehat{W} \subset \Xo_1 \times \widetilde{W}_1 \times \widehat{W}_1$ containing $(x_0, \tilde{w}_0, \hat{w}_0)$ such that for all $x \in \Xo$ the measure $\nu(x, \cdot)$ defined by
$$
\nu(x, \cdot) : A \in \borel(\del{\R^n}\new{\Xo_1}) \mapsto \int_{\widetilde{W}} \int_{\widehat{W}} \1_{A}(G(x,\tilde{w},\hat{w})) d\tilde{w} d\hat{w}
$$
is equivalent to Lebesgue measure on an open set $\Ro_x$.
\item there exists $c>0$ and open sets $\Uo_{x_0}$ and $\Vo_{x_0}^{(\tilde{w}_0, \hat{w}_0)}$ containing $x_0$ and $G(x_0, \tilde{w}_0, \hat{w}_0)$ respectively, such that for all $x \in \Xo$ and $A \in \borel(\new{\Xo_1}\del{\R^n})$
$$
\nu(x, A) \geq c\1_{\Uo_{x_0}}(x) \mleb(A \cap \Vo_{x_0}^{(\tilde{w}_0, \hat{w}_0)})
$$
\end{enumerate}
\end{lemma}

\mathnote{
Since $G$ is $C^1$ and that $[\partial G / \partial \hat{w}]$ is invertible at $(x_0, \tilde{w}_0, \hat{w}_0)$, there exists $\Oo$ an open  neighborhood of $(x_0, \tilde{w}_0, \hat{w}_0)$ such that for all $(x, \tilde{w}, \hat{w}) \in \Oo$, $\det[\partial G / \partial \hat{w} ]$ is non-zero at $(x, \tilde{w}, \hat{w})$. With the inverse function theorem applied to the function $G^*$ defined through
$$
G^* : (x, \tilde{w}, \hat{w})  \in \Oo \mapsto \left( \begin{array}{c} x \\ u \\ G(x,u,w) \end{array} \right) ,
$$
there exists $\Oo_1$ and $\Ro_1$ open neighborhoods of respectively $(x_0, \tilde{w}_0, \hat{w}_0)$ and of $G^*(x_0, \tilde{w}_0, \hat{w}_0)$ and a $C^1$ function $H^*$ such that $G^*$ is a bijection from $\Oo_1$ to $\Ro_1$ with inverse $H^*$, and
$$
\left[ \frac{\partial G^*}{\partial \hat{w}} \right] = \left[ \frac{\partial H^*}{\partial \hat{w}} \right]^{-1} .
$$
Since $\Oo_1$ is an open neighborhood of $(x_0, \tilde{w}_0, \hat{w}_0)$, there exists $\Xo$, $\tilde{\Wo}$ and $\hat{\Wo}$ such that $\Xo \times \tilde{\Wo} \times \hat{\Wo}$ is an open neighborhood of $(x_0, \tilde{w}_0, \hat{w}_0)$. Let $\Ro$ be the set $G^*(\Xo \times \tilde{\Wo} \times \hat{\Wo})$. Since the inverse of $G^*$ is continuous, $\Ro$ is an open set. Let $H$ be the composition of $H^*$ by a projection such that $H(x, \tilde{w}, G(x, \tilde{w}, \hat{w})) = G^*(x, \tilde{w}, \hat{w}) = \hat{w}$.
 As $\det [\partial G / \partial \hat{w}] > 0$ for all $(x, \tilde{w}, \hat{w}) \in \Xo \times \tilde{\Wo} \times \hat{\Wo} \subset \Oo_1 \subset \Oo$, there exists $h_0$ such that $\det [\partial H / \partial \hat{w}] > h_0$ for all $(x, \tilde{w}, v) \in \Ro$. The rest follow with a change of variables :).
}


The proof of Proposition~\ref{pr:2.1} $\mathrm{(i)}$ differs from the proof in \cite[Theorem~2.1]{meyn-caines1991} in \eqref{eq:p0} where the inequation holds for $y \in B_\delta(x_0, \delta)$ (instead of for all $y \in \xdom$). This does not impact the rest of the proof.\TODO{I am lost then why do we do a proof?}\alex{sake of completeness?} The detailed proof is given below. The proof of \cite[Theorem~2.1]{meyn-caines1991} uses Sard's theorem to show that weak stochastic controllability implies that for all $x \in \xdom$, there exists $k \in \Zplusstrict$ and $\w \in \Oxk$ such that the rank of $\Cxk(\w)$ is $n$. We use here the same idea to show $\mathrm{(ii)}$, with a slight precision to show that $\w$ may be taken in the open set $\Vo$ of \eqref{eq:wsc}.\\

We are now ready to start the core of the proof of Proposition~\ref{pr:2.1}. 
Suppose that the controllability matrix $\Cxk[k][x_0](\w_0)$ has rank $n$ for some $x_0 \in \xdom$, $k \in \Zplusstrict$ and $\w_0 \in \Oxk[k][x_0]$. Since the function $(x, \w) \mapsto \pxk(\w)$ is lower semi-continuous (as a consequence of Lemma~\ref{lm:plsc}) and that $\pxk[k][x_0](\w_0) > 0$, there exists $p_0 >0$ and $\delta$, 
such that $\pxk(\w) > p_0$ for all $(x,\w) \in B(x_0, \delta) \times B(\w_0,\delta)$. Hence
\begin{align} \label{eq:p0}
P^k(y, A) &= \int_{\Oxk[k][y]} \1_A(\Sxk[k][y](\w)) \pxk[k][y](\w) d\w \nonumber \\
 	& \geq p_0 \int_{B(\w_0, \delta)} \1_A(\Sxk[k][y](\w)) d\w  \textrm{ for all } y \in B(x_0, \delta) .
\end{align}

Since $\rank(\Cxk[k][x_0](\w_0)) = n$, we can extract a sequence of $n$ integers $(i_1, \ldots, i_n)$ for which
$$
\det \left[ \frac{\partial \Sxk[k][x_0]}{\partial w_{i_1}} | \ldots | \frac{\partial \Sxk[k][x_0]}{\partial w_{i_n}} \right]_{\w_0} \neq 0 .
$$
Hence we can apply Lemma \ref{lm:3.0} to the measure $\nu(y,A)=\int_{B(\w_0, \delta)} \1_A(\Sxk[k][y](\w)) d\w$ by appropriately defining $G$ in terms of $\Sxk$ (separating $\w \in B(\w_0, \delta)$ into $\hat{\w} = (w_{i_1}, \ldots, w_{i_n})$ and $\tilde{\w}$ equal to the other coordinates, and defining $G : (x, \tilde{\w}, \hat{\w}) \mapsto \Sxk(\w)$ which ensures that $\partial G / \partial \hat{\w}(x_0, \tilde{\w_0}, \hat{\w_0})$ has rank $n$). This shows the existence of open sets $\Uo_{x_0}$ and $\Vo_{x_0}^{\w_0}$ containing $x_0$ and $\Sxk[k][x_0](\w_0)$ respectively, an open set $\Ro_{x_0}$ and a constant $c>0$ such that
\begin{align}
&  \int_{B(\w_0, \delta)} \1_A(\Sxk[k](\w)) d\w \geq c \1_{\Uo_{x_0}}(y) \mleb(A \cap \Vo_{x_0}^{\w_0})  \textrm{ for all } y \in \xdom . \label{eq:smallsetdom}
\end{align}
Combining \eqref{eq:p0} and \eqref{eq:smallsetdom} shows that for all $y$ in the open set $\Uo_{x_0} \cap B(x_0, \delta)$,
$$
P^k(y, A) \geq c p_0 \mleb(A \cap \Vo_{x_0}^{\w_0}),
$$
which shows $\mathrm{(i)}$.

Now suppose that $F$ is $C^\infty$, hence according to Lemma~\ref{lm:scp} the function $(x, \w) \mapsto \Sxk(\w)$ is also $C^\infty$ for all $k \in \Zplusstrict$. Take $k \in \Zplusstrict$, $c > 0$ and $\Vo \in \borel(\xdom)$ an open set for which \eqref{eq:wsc} holds. Let $N$ denote the set of critical values of $\Sxk$, i.e.\ $$N := \lbrace \Sxk(\w) \in \xdom | \w \in \Oxk \textrm{ and } \rank(\Cxk(\w)) < n  \rbrace .$$
By Sard's Theorem, $\mleb(N) = 0$. Hence with \eqref{eq:wsc} for $A = \Vo \backslash N$, $P^k(x, \Vo \backslash N) \geq c \mleb( \Vo \backslash N \cap \Vo) = c \mleb(\Vo) > 0$. Hence there exists $\w \in \Oxk$ such that $\Sxk(\w) \in \Vo \backslash N$ (otherwise we would have $P^k(x, \Vo \backslash N) = 0$). And since $\Sxk(\w) \notin N$, the controllability matrix $\Cxk(\w)$ has rank $n$.
\qed

\subsubsection*{Proof of Proposition~\ref{pr:support}}

We first prove that a point $\xstar$ of the support of $\psi$ is globally attracting. By definition of the support, there exists $N$ a closed set containing $\xstar$ with full $\psi$-measure (i.e.\ $\psi(N^c) = 0$). Let $\Uo$ be an open neighborhood of $\xstar$, then $\psi(\Uo) > 0$ (otherwise the closed set $N \backslash \Uo$ would have full $\psi$-measure without containing $\xstar$, so $\xstar$ would not be in the support of $\psi$ which is a contradiction). This imply that for all $y \in \xdom$, $\sum_{k \in \Zplusstrict} P^k(y, \Uo) > 0$, and so by Proposition~\ref{pr:pointproba} that there exists $k \in \Zplusstrict$ and $\w \in \Oxk$ a $k$-steps path from $y$ to $\Uo$. Therefore by Proposition~\ref{pr:GA}, $\xstar$ is globally attracting.

Conversely, the fact that a globally attracting point $\xstar \in \xdom$ is part of $\supp \psi$ is a direct consequence of Propositions~\ref{pr:GA} and \ref{pr:pointproba}, which shows that \eqref{eq:suppGA} holds.

We now prove that $\supp \psi \subset \overline{A_+(\xstar)}$ for $\xstar \in \xdom$ a globally attracting state. Take $\ystar \in \xdom$ a globally attracting state, by Proposition~\ref{pr:GA}, $\ystar \in \overline{A_+(\xstar)}$ which implies that the set $\lbrace z^* \in \xdom | z^* \textrm{ is globally attracting} \rbrace$ is included in $\overline{A_+(\xstar)}$ which itself as we have shown equals $\supp \psi$.

Finally we prove that $\overline{A_+(\xstar)} \subset \supp \psi$ by showing that any $y \in \overline{A_+(\xstar)}$ is globally attracting. Indeed if this holds, then $\overline{A_+(\xstar)} \subset \lbrace z^* \in \xdom | z^* \textrm{ is globally attracting} \rbrace$ which with \eqref{eq:suppGA} concludes the proof. For all $z \in \xdom$ by Proposition~\ref{pr:GA} for all $\Uo$ open neighborhood of $\xstar$ there exists $t_0 \in \Zplusstrict$ and $\u \in \Oxk[t_0][z]$ a $t_0$-steps path from $z$ to $\Uo$. For $y \in \overline{A_+(\xstar)}$ for all $\Vo$ open neighborhood of $y$ there exists $k\in \Zplus$ and $y_k \in A_+^k(\xstar)$ such that $y_k \in \Vo$. If $k = 0$, then $\xstar \in \Vo$ so by taking $\Uo$ small enough to be contained in $\Vo$, $\u$ is a $t_0$-steps path from $z$ to $\Vo$. Else there exists $\w_k \in \Oxk[k][\xstar]$ for which $\Sxk[k][\xstar](\w_k) = y_k$. By continuity of the function $(x, \w) \mapsto \Sxk(\w)$ (implied by the continuity of $F$, lower semi-continuity of the function $(x,w) \mapsto p_x(w)$ and Lemma~\ref{lm:scp}) there exists $\Ro$ an open neighborhood of $\xstar$ such that for $x \in \Ro$, $\Sxk(\w_k)$ is close enough to $y_k$ to be in $\Vo$. Hence by taking $\Uo$ small enough to be contained in $\Ro$, $(\u, \w_k)$ is a $(t_0+k)$-steps path from $z$ to $\Vo$. In any case, for all $z \in \xdom$ and $\Vo$ neighborhood of $y$, there exists for some $T \in \Zplusstrict$ a $T$-steps path from $z$ to $\Vo$, which by Proposition~\ref{pr:GA} implies that $y$ is globally attracting. 
\qed

\subsubsection*{Proof of Theorem~\ref{th:ap}}
Assume that for all $x \in \xdom$, there exists $k \in \Zplusstrict$ and $\w \in \Oxk$ for which $\rank \Cxk(\w) = n$, and assume that $\xstar$ is steadily attracting; we will prove that $\markov$ is $\varphi$-irreducible and aperiodic. By hypothesis there exists $k \in \Zplusstrict$ and $\wstar \in \Oxk[k][\xstar]$ such that $\rank(\Cxk[k][\xstar](\wstar)) = n$.
According to Proposition~\ref{pr:2.1} $\mathrm{(i)}$, there exists $c > 0$ and open sets $\Uo_{\xstar}$ and $\Vo_{\xstar}^{\wstar}$ containing respectively $\xstar$ and $\Sxk[k][\xstar](\wstar)$ and such that
\begin{equation} \label{eq:xstarPsup}
P^k(y, A) \geq c \mleb( A \cap \Vo_{\xstar}^{\wstar})  \textrm{ for all } y \in \Uo_{\xstar},  A \in \borel(\xdom).
\end{equation}
For all $y \in \xdom$ and $t \in \Zplusstrict$ we may develop $P^{t+k}(y, A)$ as
\begin{align} \label{eq:Psupapp}
P^{t + k}(y, A) &= \int_{\xdom} P^{t}(y, dz) P^k(z, A) \nonumber \\
 &\geq \int_{\Uo_{\xstar}} P^{t}(y, dz) P^k(z, A) \nonumber  \\
 &\geq \int_{\Uo_{\xstar}} P^{t}(y, dz) c \mleb(A \cap \Vo_{\xstar}^{\wstar}) =  P^{t}(y, \Uo_{\xstar}) c \mleb(A \cap \Vo_{\xstar}^{\wstar}) 
\end{align}
As $\xstar$ is globally attractive, \del{that conditions A1---A4 hold and that $F$ is continuous,} by Corollary~\ref{cr:GAreachability} $\xstar$ is also reachable. Hence for all $y \in \xdom$ there exists a $t \in \Zplusstrict$ such that $P^t(y, \Uo_{\xstar}) > 0$, and therefore by \eqref{eq:Psupapp} the measure $\bar{\varphi} : A \mapsto \mleb(A \cap \Vo_{\xstar}^{\wstar})$ is an irreducibility measure. 

Since the Markov chain $\markov$ is $\varphi$-irreducible, according to \cite[Theorem~5.4.4]{meyntweedie} there exists $d \in \Zplusstrict$ and disjoint sets $(D_i)_{i = 0, \ldots, d-1} \in \borel(\xdom)^d$ such that 
\begin{enumerate}
\item[(a)] for $x \in D_i$, $P(x, D_{i+1 \mod d}) = 1  , ~~i = 0, \ldots, ~d-1$ (mod d) ;
\item[(b)] $\varphi\left( \left( \bigcup_{i=0}^{d-1} D_i \right)^c \right) = 0$.
\end{enumerate}
Note that $\mathrm{(b)}$ is usually stated with respect to the maximal irreducibility measure, which implies $\mathrm{(b)}$ for any irreducibility measure. Point $\mathrm{(b)}$ applied to $\bar{\varphi}$ implies the existence of  $i \in \lbrace 0, \ldots, d-1\rbrace$ for which $\bar{\varphi}(D_i) > 0$, i.e.\ $\mleb(D_i \cap \Vo_{\xstar}^{\wstar}) > 0$.\mathnote{This will actually hold for any $D_j$ by globaly attractivity contamination.} Also since $\xstar$ is steadily attracting, for all $y \in \xdom$ there exists $t_0 \in \Zplusstrict$ such that for all $t \geq t_0$, there exists a $t$-steps path from $y$ to $\Uo_{\xstar}$ and so by Proposition~\ref{pr:pointproba} $P^{t}(y, \Uo_{\xstar}) > 0$. Together with \eqref{eq:Psupapp}, we obtain $P^{t + k}(y, D_i) \geq P^{t}(y, \Uo_{\xstar}) c \mleb(D_i \cap \Vo_{\xstar}^{\wstar}) > 0$ for all $t \geq t_0$. Hence for $y \in D_i$ and for $m$ large enough, $P^{md + 1}(y, D_{i}) > 0$, which with $\mathrm{(a)}$ and the fact that the $D_i$ are disjoints sets implies $d = 1$, i.e.\ $\markov$ is aperiodic.

Now suppose that no steadily attracting state exists, we will show that $\markov$ is not irreducible and aperiodic. If no globally attracting state exists, then by Theorem~\ref{th:irr} the Markov chain is not irreducible. Otherwise let $\xstar$ denote a globally attracting state, and take $k \in \Zplusstrict$ and $\wstar \in \Oxk[k][\xstar]$ for which $\rank \Cxk[k][\xstar](\wstar) = n$. Then according to Proposition~\ref{pr:GAEA}, the point $\ystar := \Sxk[k][\xstar](\wstar)$ is attainable, and according to Theorem~\ref{th:irr-practical}, $\markov$ is $\varphi$-irreducible. Since $\ystar$ is attainable yet not steadily attracting, by Proposition~\ref{pr:EA2AA}, the set
$$E := \lbrace a \in \Zplusstrict | \exists\, t_0 \in \Zplus,  \forall t \geq t_0,  \ystar \in A_+^{at}(\ystar) \rbrace$$
has a $\gcd$ larger than $1$, and there exists a $d$-cycle with $d := \gcd(E) > 1$. Hence the period of the Markov chain is at least $d$, and so $\markov$ is not aperiodic.
\qed

\subsubsection*{Proof of Theorem~\ref{th:ap-practical}}
Suppose that $\xstar$ is a steadily attracting point, and there exists $k \in \Zplusstrict$ and $\wstar \in \Oxk[k][\xstar]$ for which $\rank \Cxk[k][\xstar](\wstar) = n$. Then according to Proposition~\ref{pr:gacontrol}, for all $x \in \xdom$ there exists $t \in \Zplusstrict$ and $\w \in \Oxk[t]$ for which $\rank \Cxk[t](\w) = n$. According to Theorem~\ref{th:ap}, $\markov$ is therefore a $\varphi$-irreducible and aperiodic Markov chain, and by Theorem~\ref{th:irr-practical}, it is a $T$-chain for which compact sets are petite. Finally, according to \cite[Theorem~5.5.7]{meyntweedie} since $\markov$ is irreducible and aperiodic, every petite set is small.
\qed

\subsubsection*{Proof of Proposition~\ref{pr:xnesdens}}
Let $\U_1=(U_1^1, \ldots, U_1^{\lambda}) \in \R^{n \lambda}$ with each $U_1^i$ following a standard multivariate normal distribution and $\{U_1^i : 1 \leq i \leq \lambda \}$ i.i.d. Let $W_1=\alpha(z,\U_1)$, let $\mu$ satisfying $1 < \mu < \lambda$, $w^\mu \in \R^n$ and $C(w^\mu)$ denote the event 
$\{f(z + w^\mu ) \leq f(z + \U_1^i ), \mbox{ for all } i= \mu + 1, \ldots \lambda  \}$. The following holds
\begin{equation}\label{eq:pc}
\Pr(C(w^\mu)) = \left( \int \1_{f(z+w^\mu) \leq f(z+w)}p_{\cal N}(w) dw \right)^{\lambda - \mu} = (1 - Q_z^f(w^\mu))^{\lambda - \mu} \enspace.
\end{equation}
Let $\mathfrak{S}_\lambda$ for $\lambda \in \Zplusstrict$ denote the set of permutation with $\lambda$ elements. For $\sigma \in \mathfrak{S}_{\lambda - \mu}$, let $C_\sigma(w^\mu)$ denote the event 
$
C_\sigma(w^\mu) = \{ f(z + w^\mu) < f(z + U_1^{\mu + \sigma(1)}) < \ldots < f(z + U_1^{\mu + \sigma(\lambda - \mu)}) \}
$. Since the level sets of $f$ are Lebesgue negligible, 
\begin{equation}\label{eq:pcc}
\Pr(C(w^\mu)) = \Pr \left( \bigcup_{\sigma \in \mathfrak{S}_{\lambda - \mu}} C_\sigma(w^\mu) \right) = \sum_{\sigma \in \mathfrak{S}_{\lambda - \mu}} \!\! \Pr( C_\sigma(w^\mu)) = (\lambda - \mu)! \Pr(C_{I}(w^\mu))
\end{equation}
where $I$ denotes the identity permutation.
Let $x_1 \in \R^n, \ldots, x_\mu \in \R^n$ and let denote for $x,y \in \R^n$, $x \leq y$ for $([x]_1 \leq [y]_1 , \ldots , [x]_n \leq [y]_n)$ where $[x]_i$ is the $i^{\rm th}$ coordinate of the vector $x$.
\begin{equation*}
\begin{split}
\Pr( W_1^1 & \leq x_1 ,  \ldots , W_1^\mu \leq x_\mu )  \\
& = \sum_{\sigma \in \mathfrak{S}_{\lambda}} 
\Pr \left( \{ U_1^{\sigma(1)} \leq x_1, \ldots, U_1^{\sigma(\mu)} \leq x_\mu \} \bigcap \{ f(z + U^{\sigma(1)}_1) < \ldots < f(z + U^{\sigma(\lambda)}_1)  \}  \right) \\
& = \lambda ! \Pr \left( \{ U_1^{1} \leq x_1, \ldots, U_1^{\mu} \leq x_\mu \} \bigcap\{ f(z + U^{1}_1) < \ldots < f(z + U^{\lambda}_1)  \}  \right) \enspace.
\end{split}
\end{equation*}
By cutting the event $  \{ f(z + U^{1}_1) < \ldots < f(z + U^{\lambda}_1)  \} $ into $\{ f(z + U^{1}_1) < \ldots < f(z + U^{\mu}_1)  \}  \bigcap\{ f(z + U_{1}^{\mu}) < \ldots < f(z + U^{\lambda}_1)  \} $ we find that
\begin{equation*}
\begin{split}
& \Pr(  W_1^1  \leq x_1 ,  \ldots , W_1^\mu \leq x_\mu )   \\
& = \lambda ! \int \!\! \1_{w^1 \leq x_1} \ldots \1_{w^\mu \leq x_\mu} \1_{f(z+w^1) < \ldots < f(z+w^\mu)} \\ & \hspace{5cm} \1_{f(z+w^\mu) < \ldots < f(z+w^\lambda)}  p_{\cal N}(w^1) \! \ldots  \! p_{\cal N}(w^\lambda) dw^1 \! \! \ldots dw^\lambda\\
&= \lambda ! \int \1_{w^1 \leq x_1} \ldots \1_{w^\mu \leq x_\mu} \1_{f(z+w^1) < \ldots < f(z+w^\mu)} \Pr(C_{\new{I}}(w^\mu)) p_{\cal N}(w^1) \ldots p_{\cal N}(w^\mu) dw^1 \ldots dw^\mu.
\end{split}
\end{equation*}
Using \eqref{eq:pc} and \eqref{eq:pcc} we find
\begin{equation*}
\begin{split}
& \Pr( W_1^1  \leq x_1 ,  \ldots , W_1^\mu \leq x_\mu )  \\
& = \frac{\lambda !}{(\lambda - \mu)!} \int_{-\infty}^{x_1} \!\!\!\!\! \ldots  \int_{-\infty}^{x_\mu}  \!\!\! \1_{f(z+w^1) < \ldots < f(z+w^\mu)}  (1 - Q_z^f(w^\mu))^{\lambda - \mu} p_{\cal N}(w^1) \ldots p_{\cal N}(w^\mu) dw^1 \ldots dw^\mu
\end{split}
\end{equation*}
where in the previous equation, the integrals between $-\infty$ and the vector $x_i$ stand for $\int_{-\infty}^{[x_i]_1} \ldots \int_{-\infty}^{[x_i]_n}$ where $[x_i]_j$ denotes the $j^{\rm th}$ coordinate of the vector $x_i$. We directly deduce the expression of the density for $\mu > 1$ and $\mu < \lambda$ given in \eqref{dens:mu}. 
When $\mu=\lambda$, then the density of $(W_1^1,\ldots,W_1^\lambda)$ equals $p_{\cal N}(w^1) \ldots p_{\cal N}(w^\lambda)$ that can also be written as
$$
\lambda ! \1_{f(z+w^1) < \ldots < f(z+w^\mu)}  p_{\cal N}(w^1) \ldots p_{\cal N}(w^\lambda)
$$
which is the expression in \eqref{dens:mu} for $\mu = \lambda$ (by using that the levels sets of $f$ are Lebesgue negligible, and so that with probability $1$, one of the $\lambda!$ possible \new{strict} orderings of the $\{w_i: i = 1, \ldots, \lambda\}$ according to $f(z+w^i)$ occurs).

In the case where $\mu=1$, the indicator $\1_{f(z+w^1) < \ldots < f(z+w^\mu)} $ reduces to $1$ and we deduce the expression in \eqref{dens:one}.
\qed

\subsubsection*{Proof of Lemma~\ref{lem:PCSA}}

The proof of Lemma~\ref{lem:PCSA} $\mathrm{(ii)}$ is a direct consequence of the following proposition.

\begin{proposition} \label{pr:saab}
Suppose that assumptions $\mathrm{A1-A5}$ hold and that for all $x \in \xdom$ there exists $k \in \Zplusstrict$ and $\wstar \in \Oxk[k][\xstar]$ such that $\rank \Cxk[k][\xstar](\wstar) = n$. Then there exists a steadily attracting state if and only if there exists a globally attracting state $\xstar \in \xdom$, $(a, b) \in \Zplusstrict^2$ with $\gcd(a,b) = 1$, $\w_a \in \Oxk[a][\xstar]$ and $\w_b \in \Oxk[b][\xstar]$ such that $\Sxk[a][\xstar](\w_a) = \Sxk[b][\xstar](\w_b) = \xstar$.
\end{proposition}

\begin{proof}
Suppose that $\xstar$ is a steadily attracting state. Under the conditions of Proposition~\ref{pr:saab} there exists $k \in \Zplusstrict$ and $\wstar \in \Oxk[k][\xstar]$ for which $\rank \Cxk[k][\xstar](\wstar)$, and by Proposition~\ref{pr:GAEA} $\ystar := \Sxk[k][\xstar](\wstar)$ is attainable, and there exists $\Vo_{\xstar}$ a neighborhood of $\xstar$ such that for all $x \in \Vo_{\xstar}$ there exists $\w \in \Oxk$ for which $\Sxk(\w) = \ystar$. Since $\xstar$ is steadily attracting, there exists $T \in \Zplusstrict$ such that for all $t \geq T$ there is a $t$-steps path from $\ystar$ to $\Vo_{\xstar}$, which can be completed by a $k$-steps path to $\ystar$. So for all $t \geq T$ there exists $\w \in \Oxk[t+k][\ystar]$ for which $\Sxk[t+k][\ystar](\w) = \ystar$, which proves the first implication by choosing $t_1 \geq T$ and $t_2 \geq T$ such that $\gcd(t_1 + k, t_2 +k) = 1$.

Now let $\xstar$ be a globally attracting state, $k \in \Zplusstrict$ and $\wstar \in \Oxk[k][\xstar]$ for which $\rank \Cxk[k][\xstar](\wstar) = n$. Then by Proposition~\ref{pr:GAEA} $\ystar := \Sxk[k][\xstar](\wstar)$ is an attainable point, and there exists $\Vo_{\xstar}$ a neighborhood of $\xstar$ such that for all $x \in \Vo_{\xstar}$ there exists $\w \in \Oxk$ for which $\Sxk(\w) = \ystar$. Suppose that there exists $(a,b) \in \Zplusstrict^2$, $\w_a$ and $\w_b$ such as stated in Proposition~\ref{pr:saab}. We first show that $\ystar$ is steadily attracting using Proposition~\ref{pr:EA2AA}, by showing that the set $E$ in \eqref{eq:E} has two elements $\tilde{a}$ and $\tilde{b}$ with $\gcd(\tilde{a}, \tilde{b}) = 1$.

By B\'{e}zout's identity, there exists $(c, d) \in \mathbb{Z}^2$ such that $ca + db = 1$. Suppose without loss of generality that $d < 0$ (and so $c > 0$).
By lower semi-continuity of the function $\x \mapsto \pxk[ca](\w_a, \ldots, \w_a)$ \new{and continuity of the function $x \mapsto \Sxk[ca](\w_a, \ldots, \w_a)$} there exists $\epsilon_a > 0$ such that for all $x \in B(\xstar, \epsilon_a)$, $(\w_a, \ldots, \w_a) \in \Oxk[x][ca]$ \new{and $\Sxk[c a][x](\w_a, \ldots, \w_a) \in B(\xstar, \epsilon)$}.
\del{By continuity of $F$ for all $\epsilon > 0$ there exists $\eta_a > 0$ such that for all $x \in B(\xstar, \eta_a)$, $\Sxk[c a][x](\w_a, \ldots, \w_a) \in B(\xstar, \epsilon)$ (using that $\Sxk[c a][\xstar](\w_a, \ldots, \w_a) = \xstar$).} We chose $\epsilon$ small enough to ensure that $B(\xstar, \epsilon) \subset \Vo_{\xstar}$. Since $\xstar$ is globally attracting, there exists $t\in \Zplusstrict$ and $\u \in \Oxk[t][\ystar]$ such that $\Sxk[t][\ystar](\u) \in B(\xstar, \del{\min(}\epsilon_a\del{, \eta_a)})$, and so $\Sxk[t+ca][\ystar](\u, \w_a, \ldots, \w_a) \in \Vo_{\xstar}$. Hence there exists $\w \in \Oxk[k][{\Sxk[t+ca][\ystar](\u, \w_a, \ldots, \w_a)}]$ such that $\Sxk[t+ca][\ystar](\u, \w_a, \ldots, \w_a, \w) = \ystar$. Since the path $t+ca+k$-steps path $(\u, \w_a, \ldots, \w_a, \w)$ can be retaken infinitely many times, $t+ca+k \in E$.
The same reasonning holds for $-db$, and so $t-db+k \in E$ (the same $t$ can be obtained by taking a $t$-steps path $\u$ ensuring that $\Sxk[t][\ystar](\u) \in B(\xstar, \min(\epsilon_a, \del{\eta_a,} \epsilon_b\del{, \eta_b}))$. Finally, since $(t + ca + k) - (t - db + k) = 1$, $\gcd(E) = 1$ and so by Proposition~\ref{pr:EA2AA} $\ystar$ is steadily attracting.

\end{proof}

\bibliography{biblio}
\bibliographystyle{plain}


\end{document}